\documentclass[12pt]{amsart}
 \usepackage{amsfonts,amssymb,amsthm}
\usepackage{graphicx,float} 
 \usepackage{color}
 \usepackage{array}
\usepackage[all,ps,cmtip,rotate]{xy}

\usepackage{mathrsfs}

\usepackage{bm}

\usepackage{hyperref}

   \allowdisplaybreaks

\def\Z{{\bf Z}}

\def\P{{\bf P}}
\def\PP{{\bf P}}

\def\cI{\mathscr{I}}

\def\cA{\mathscr{A}}

\def\cF{\mathscr{F}}

\def\cO{\mathscr{O}}

\def\cE{\mathscr{E}}

\def\cC{\mathscr{C}}

\def\cU{\mathscr{U}}
\def\cV{\mathscr{V}}

\def\cX{X^5}
\def\Xf{X^4}

\def\bQ{\mathbf{Q}}

\def\k{\mathbf k}

\def\bq{\mathbf q}

\def\lra{\longrightarrow}
\def\llra{\hbox to 10mm{\rightarrowfill}}
\def\lllra{\hbox to 15mm{\rightarrowfill}}

\def\llla{\hbox to 10mm{\leftarrowfill}}
\def\lllla{\hbox to 15mm{\leftarrowfill}}

\def\isom{\simeq}

\def\emptyset{\varnothing}

\DeclareMathOperator{\isomlra}{\stackrel{{}_{\scriptstyle\sim}}{\lra}}

\DeclareMathOperator{\isomto}{\isomlra}

\DeclareMathOperator{\codim}{codim}

\DeclareMathOperator{\Coker}{Coker}

\DeclareMathOperator{\GL}{GL}
\DeclareMathOperator{\Gr}{\mathsf{Gr}}
\DeclareMathOperator{\LGr}{\mathsf{SGr}}

\DeclareMathOperator{\Fl}{\mathsf{Fl}}
\DeclareMathOperator{\LFl}{\mathsf{SFl}}

\DeclareMathOperator{\Ker}{Ker}

\DeclareMathOperator{\Pf}{Pf}
\DeclareMathOperator{\PGL}{PGL}
\DeclareMathOperator{\Pic}{Pic}
\DeclareMathOperator{\pr}{\mathsf{pr}}

\DeclareMathOperator{\Sp}{Sp}

\DeclareMathOperator{\rank}{rank}

\DeclareMathOperator{\Sym}{\mathsf S}
\DeclareMathOperator{\SL}{SL}

\def\bw#1#2{\textstyle{\bigwedge\hskip-0.9mm^{#1}}\hskip0.2mm{#2}}
\def\bww#1#2#3{\textstyle{\bigwedge\hskip-0.9mm^{#1}_{#2}}\hskip0.2mm{#3}}

\newtheorem{lemma}{Lemma}[section]
\newtheorem{theorem}[lemma]{Theorem}
\newtheorem{corollary}[lemma]{Corollary}
\newtheorem{proposition}[lemma]{Proposition}

\theoremstyle{remark}

\newtheorem{remark}[lemma]{Remark}

\def\BU{{\overline{U}}}
\def\BW{{\overline{W}}}

\def\one{{\mathbf{1}}}

\def\setminus{\smallsetminus}
\def\cong{\isom}

\newcommand{\conv}{\,\lrcorner\,}

\newcommand{\Gtwo}{{\mathbb{G}_2}}

\newcommand{\blam}{{\bar{\lambda}}}
\newcommand{\hlam}{{\hat{\lambda}}}
\newcommand{\hxi}{{\hat{\xi}}}

\newcommand{\tcX}{\widetilde{X}^5}
\newcommand{\bcX}{\overline{X}^5}
\newcommand{\tX}{\widetilde{X}^4}
\newcommand{\bX}{\overline{X}^4}
\newcommand{\bpi}{{\bar\pi}}
\newcommand{\bi}{{\bar{\imath}}}

\newcommand{\bp}{{\bar{p}}}
\newcommand{\be}{{\bar{e}}}
\newcommand{\bh}{{\bar{h}}}
\newcommand{\BD}{{\overline{D}}}
\newcommand{\BE}{{\overline{E}}}
\newcommand{\BC}{{\overline{C}}}
\newcommand{\bGamma}{{\overline{\Gamma}}}

\newcommand{\sv}{{\mathsf{v}}}

\DeclareMathOperator{\rk}{{\mathrm{rank}}}

\DeclareMathOperator{\Bl}{{\mathrm{BL}}}
\DeclareMathOperator{\CH}{{\mathrm{CH}}}
\DeclareMathOperator{\Mot}{{\mathbb{M}}}
\DeclareMathOperator{\Lef}{{\mathbb{L}}}

\newcommand{\bcU}{{\overline{\cU}}}

\newcommand{\lm}{{\lambda,\mu}}
\newcommand{\lmn}{{\lambda,\mu,\nu}}

\begin{document}

\title{K\"uchle fivefolds of type c5}

 \author[A. Kuznetsov]{Alexander Kuznetsov}
 \address{Algebra Section, Steklov Mathematical Institute,
  8 Gubkin str., Moscow 119991 Russia}
 \email{{\tt  akuznet@mi.ras.ru}}
 
\thanks{This work is supported by the Russian Science Foundation under grant 14-50-00005.}

\begin{abstract}
We show that K\"uchle fivefolds of type $(c5)$ --- subvarieties of the Grassmannian $\Gr(3,7)$ parameterizing 3-subspaces that are isotropic for a given 2-form and are annihilated by a given 4-form --- 
are birational to hyperplane sections of the Lagrangian Grassmannian $\LGr(3,6)$ and describe in detail these birational transformations.
As an application, we show that the integral Chow motive of a K\"uchle fivefold of type $(c5)$ is of Lefschetz type.
We also discuss K\"uchle fourfolds of type $(c5)$ --- hyperplane sections of the corresponding K\"uchle fivefolds ---
an interesting class of Fano fourfolds, which is expected to be similar to the class of cubic fourfolds in many aspects. 
\end{abstract}

\maketitle


\section{Introduction}

One of the most interesting classical questions of birational geometry is the question of rationality of cubic fourfolds.
In the Italian school of algebraic geometry it was believed \cite{morin1940} that a general cubic fourfold is rational, but the argument had a gap. 
Now some families of rational cubic fourfolds are known \cite{hassett2000special} (such as Pfaffian cubics, and some cubics containing a plane),
but it is generally believed that a very general cubic fourfold is irrational. 
However, in spite of many attempts, proving irrationality of a single cubic fourfold remains out of reach.

One of the points of view on (ir)rationality of cubic fourfolds is via the structure of their derived categories (see~\cite{kuznetsov2015rationality}). 
It is known that the derived category of coherent sheaves on a cubic fourfold~$X$ has a semiorthogonal decomposition consisting of three exceptional objects and an additional
subcategory $\cA_X$, whose properties resemble very much those of the derived category of a K3 surface (such categories are usually called noncommutative K3 surfaces). 
It is conjectured (see~\cite{kuznetsov2010cubic,kuznetsov2015rationality}) that $X$ is rational if and only if $\cA_X$ is equivalent to the derived category of a (commutative) K3 surface.
This conjecture is consistent with all known examples of rational cubic fourfolds --- for those $\cA_X$ is equivalent 
to the derived category of a K3 surface, while for a very general $X$ it is easy to show that $\cA_X$ is not equivalent 
to the derived category of any K3 surface. 

While it is not clear how the above conjecture could be proved, it is quite interesting to investigate other families of fourfolds 
which have similar properties (i.e.\ whose derived categories contain noncommutative K3 surfaces as semiorthogonal components). 
One of such families, Gushel--Mukai fourfolds was investigated in \cite{debarre2015gushel} and  \cite{kuznetsov2016gushel}. 
This paper makes a first step to investigation of yet another family of fourfolds with similar properties.



In 1995 Oliver K\"uchle classified in~\cite{kuchle1995fano} all Fano fourfolds of index 1 that can be obtained as zero loci of regular global sections 
of equivariant vector bundles on Grassmannians. The list of such fourfolds includes 20 families (in fact, originally there were 
21 families in the list, but two of them were recently shown to be equivalent \cite{manivel2015}), and three of them --- types $(c7)$, $(d3)$, and $(c5)$ in K\"uchle's notation --- 
judging by their Hodge numbers, might contain a noncommutative K3 surface as a component of their derived category.
The first two types were considered in~\cite{kuznetsov2015kuchle}, and were shown not to produce an interesting example.
A fourfold of type $(d3)$ was shown to be isomorphic to the blowup of $(\P^1)^4$ with center in a K3 surface, 
and a fourfold of type $(c7)$ to the blowup of a cubic fourfold with center in a Veronese surface. 
However, we expect that the last of the three examples --- a fourfold of type $(c5)$ --- is new and interesting.

By definition such fourfolds can be constructed as follows. Consider the Grassmannian $\Gr(3,7)$ of 3-dimensional
vector subspaces in a 7-dimensional vector space. Let $\cU_3$ and $\cU_3^\perp$ be the tautological vector subbundles 
on the Grassmannian, of ranks 3 and 4 respectively. Consider the following rank 8 vector bundle 
\begin{equation*}
\cU_3^\perp(1) \oplus \cU_3(1) \oplus \cO(1).
\end{equation*}
Its global section is given by a triple $(\lambda,\mu,\nu)$, where $\lambda$ is a 4-form, $\mu$ is a 2-form, and $\nu$ is a 3-form
on the 7-dimensional vector space. The zero locus of such a section (provided it is sufficiently general) is a smooth Fano fourfold
\begin{equation*}
X = X^4_{\lambda,\mu,\nu} \subset \Gr(3,7).
\end{equation*}
Its numerical invariants, computed by K\"uchle, are 
\begin{equation*}
K_X^4 = 66,
\qquad
h^0(X,\cO(-K_X)) = 20,
\end{equation*}
and its Hodge diamond looks as
\begin{equation*}
\begin{smallmatrix}
&&&& 1 \\
&&& 0 && 0 \\
&& 0 && 1 && 0 \\
& 0 && 0 && 0 && 0 \\
0 && 1 && 24 && 1 && 0 \\
& 0 && 0 && 0 && 0 \\
&& 0 && 1 && 0 \\
&&& 0 && 0 \\
&&&& 1 
\end{smallmatrix}
\end{equation*}
In particular, the Hodge diamond of a K3 surface is clearly seen in its center, so one can expect to find a noncommutative K3 category as a component of its derived category.
Of course, to prove something of this sort, we need to understand the geometry of this variety better. The goal of this paper is to do some steps in this direction.

Our approach is based on the following funny common feature of the three examples of Fano fourfolds, which have (or might have) a noncommutative K3 surface.
In fact, all of them are half-anticanonical sections of nice Fano fivefolds. The corresponding fivefolds are $\P^5$, a hyperplane section of $\Gr(2,5)$, 
or the zero locus of the section $(\lambda,\mu)$ of the vector bundle $\cU_3^\perp(1) \oplus \cU_3(1)$ on $\Gr(3,7)$, which we denote by $\cX_{\lambda,\mu}$
(and call {\sf K\"uchle fivefolds of type $(c5)$}).
\begin{equation*}
\begin{array}{|l|c|c|c|}
\hline
\text{Fano fourfold} & \text{cubic fourfold} & \text{Gushel--Mukai fourfold} & \text{K\"uchle fourfold $X_{\lambda,\mu,\nu}$} \\
\hline 
\text{Fano fivefold} & \P^5 & \Gr(2,5) \cap H & \cX_{\lambda,\mu} \\
\hline
\end{array}
\end{equation*}
The structure of the derived category of the fourfolds (at least of cubic and Gushel--Mukai fourfolds) is determined by the structure of the derived category of the corresponding fivefolds. 
Both in case of $\P^5$ and $\Gr(2,5) \cap H$, the derived category has a rectangular Lefschetz decomposition with respect to (a fraction of) the half-anticanonical line bundle (see~\cite{kuznetsov2014icm}). 
From this an existence of a noncommutative K3 category in a fourfold follows by \cite{kuznetsov2015calabi}. 
By the way, two ``non-interesting'' examples $(d3)$ and $(c7)$ also share this feature --- the corresponding fivefolds are $(\P^1)^5$ and the blowup of $\P^5$ 
with center in the Veronese surface both have a rectangular Lefschetz decomposition, see~\cite{kuznetsov2015kuchle}.

It is natural to expect that the same is true for K\"uchle fivefolds of type $(c5)$.
An attempt to construct a rectangular Lefschetz decomposition of $\cX_\lm$ was the main motivation for this paper.
Although we have not succeeded in this yet, a geometrical construction we have found, allowed us to show that the Chow motive (with integral coefficients) of a general K\"uchle fivefold is of Lefschetz type. 
This can be considered as an approximation to the derived category result we are up to.



Let us explain this geometrical construction.
Recall that by definition, $\cX_\lm$ is a subvariety in the Grassmannian $\Gr(3,7)$ defined as the zero locus of a section $(\lm)$ 
of the vector bundle $\cU_3^\perp(1) \oplus \cU_3(1)$ given by a 4-form $\lambda$ and a 2-form~$\mu$.
We associate with it a certain hyperplane section $\LGr(3,6) \cap H$ of the Lagrangian Grassmannian $\LGr(3,6)$ 
and a codimension 2 subvariety $Z \subset \LGr(3,6) \cap H$, that is isomorphic to a scroll (a $\P^1$-bundle) over a sextic del Pezzo surface.
Then we show that the blowup $\tcX_\lm$ of $\LGr(3,6) \cap H$ with center in $Z$ can be also realized as a blowup of~$\cX_\lm$.
Moreover, we show that the center $F$ of the blowup $\tcX_\lm \to \cX_\lm$ is isomorphic to the flag variety $\Fl(1,2;3)$.
In other words, we have a simple birational transformation between $\cX_\lm$ and a hyperplane section of the Lagrangian Grassmannian that can be expressed by a diagram
\begin{equation*}
\xymatrix{
&& \tcX_\lm \ar[dl] \ar[dr] \\
F \ar@{^{(}->}[r] & \cX_\lm && \LGr(3,6) \cap H & Z \ar@{_{(}->}[l] 
}
\end{equation*}
%
%
%
%
Everything in this construction can be described quite explicitly, see section~\ref{section:5folds} for more details.
We believe this description should be essential for understanding the geometry of K\"uchle fivefolds $\cX_\lm$ and their hyperplane sections $\Xf_\lmn$.
We demonstrate its usefulness by applying it to the computation of the Chow motive of $\cX_\lm$ in section~\ref{section:apps}.

The paper is organized as follows.
In section~\ref{section:preliminaries} we introduce some notation and prove a very basic, but rather useful result (a blowup lemma)
allowing in some cases to identify a subscheme in a projective bundle as a blowup of its base.
In section~\ref{section:forms-and-generality} we discuss the geometry of 3-forms on a 6-space and of 4-forms on a 7-space.
After that we introduce explicit generality assumptions on a pair $(\lm)$ of a 4-form and a 2-form on a 7-space, under which we work later.
We explain the standard form, in which such a pair can be written, and introduce some useful geometric constructions related to this data.
Section~\ref{section:5folds} is the main part of the paper. 
Here we construct the birational transformation between a K\"uchle 5-fold $\cX_\lm$ and its associated hyperplane section of the Lagrangian Grassmannian $\LGr(3,6)$, and discuss the details of its geometry.
In section~\ref{section:apps} we give two applications of this description. 
First, we show that a general K\"uchle fourfold $\Xf_\lmn \subset \cX_\lm$ is birational to a singular (along a curve) quadratic section of the hyperplane section of the Lagrangian Grassmannian.
Second, we show that the integral Chow motive of a K\"uchle fivefold is of Lefschetz type.
Finally, in section~\ref{section:hyperplane-sgr} we prove that the integral Chow motive of any smooth hyperplane section of the Lagrangian Grassmannian $\LGr(3,6)$ is of Lefschetz type.
For this we introduce a new geometric construction --- we identify a certain $\P^2$-bundle over such a hyperplane section 
with the blowup of the isotropic Grassmannian $\LGr(2,6)$ with center in the adjoint variety of the simple algebraic group of type $\Gtwo$.

\subsection*{Acknowledgements:}
I am very grateful to Atanas Iliev, Grzegorz and Micha\l{}  Kapustka, Laurent Manivel, Dmitri Orlov, and Kristian Ranestad for useful discussions.



\section{Preliminaries}\label{section:preliminaries}

\subsection{Notations and conventions}\label{subsection:notations}

We work over an algebraically closed field $\k$ of characteristic~0.
For any vector space $V$ we denote by $\wedge$ the wedge product of skew forms and polyvectors and by $\conv$ the convolution operation
\begin{equation*}
\bw{p}V \otimes \bw{q}V^\vee \xrightarrow{\ \conv\ } \bw{p-q}V
\qquad\qquad
\text{(if $p \ge q$),}
\end{equation*}
induced by the natural pairing $V \otimes V^\vee \to \k$. 

If $p = n = \dim V$ and $0 \ne \epsilon \in \det(V)$, the convolution with $\epsilon$ gives an isomorphism
\begin{equation}\label{eq:isomorphism-convolution}
\bw{q}V^\vee \isomto \bw{n - q}V,
\qquad 
\xi \mapsto \xi^\vee := \epsilon \conv \xi.
\end{equation}
This isomorphism is canonical up to rescaling (since $\epsilon$ is unique up to rescaling). Note that
\begin{equation}\label{eq:convolution-duality}
\omega \conv \xi = (\xi^\vee) \conv (\omega^\vee),
\end{equation} 
where $\omega \in \bw{k}V$, $\xi \in \bw{p}V^\vee$ and $k \ge p$, and we use $\epsilon^{-1} \in \det(V^\vee)$ to define $\omega^{-1}$.

We say that a $p$-form $\xi$ {\sf annihilates} a $k$-subspace $U \subset V$, if $k \le p$ and $\xi \conv (\bw{k} U) = 0$.
Analogously, we say that $U \subset V$ is {\sf isotropic} for a $p$-form $\xi$ if $p \le k$ and $(\bw{k} U) \conv \xi = 0$.
By~\eqref{eq:convolution-duality} a subspace $U \subset V$ is isotropic for $\xi$ if and only if $\xi^\vee$ annihilates $U^\perp := \Ker(V^\vee \to U^\vee)$.

For a vector space $V$ we denote by $\P(V)$ the projective space of one-dimensional subspaces in $V$ and by $\cO(1)$ the very ample generator of its Picard group, so that $H^0(\P(V),\cO(1)) = V^\vee$.
Abusing the notation, we frequently consider nonzero vectors $v \in V$ as points of $\P(V)$ and vice versa. 
In case we want to emphasize a difference, we denote the point, corresponding to a vector $v \in V$ by $[v] \in \P(V)$, and the corresponding one-dimensional subspace by $\k v \subset V$.

Analogously, for a vector bundle $\cV$ on a scheme $S$ we denote by $\pr:\P_S(\cV) \to S$ the projective bundle, parameterizing one-dimensional subspaces in the fibers of $\cV$, 
and by $\cO(1)$ the ample generator of the relative Picard group such that $\pr_*\cO(1) \cong \cV^\vee$.
Sometimes, we refer to the divisor class of this line bundle as the {\sf relative hyperplane class}.
Note that although the projectivization of a vector bundle does not change if the vector bundle get twisted, the corresponding relative hyperplane class does.

We always denote by $W$ a vector space of dimension~7, and by $\BW$ a vector space of dimension 6.
In fact, further on the space $\BW$ will be a direct summand of~$W$, but for a moment this is irrelevant. 
We use notation $e_0,e_1,\dots,e_6$ for a basis in $W$ and $e_1,\dots,e_6$ for a basis in~$\BW$. 
The dual bases in $W^\vee$ and $\BW^\vee$ are denoted by $x_0,x_1,\dots,x_6$ and $x_1,\dots,x_6$ respectively.
We usually abbreviate $x_{i_1} \wedge \dots \wedge x_{i_p}$ to $x_{i_1\dots i_p}$ and $e_{i_1} \wedge \dots \wedge e_{i_p}$ to $e_{i_1\dots i_p}$. 

We denote by $\Gr(k,W)$ the Grassmannian of $k$-dimensional vector subspaces in $W$. The tautological vector subbundle of rank $k$ on it
is denoted by $\cU_k \subset W \otimes \cO_{\Gr(k,W)}$. The quotient bundle is denoted simply by $W/\cU_k$, and for its dual we use the notation
\begin{equation*}
\cU_k^\perp := (W/\cU_k)^\vee.
\end{equation*}
Analogously, we denote by $\bcU_k$ the tautological subbundle on $\Gr(k,\BW)$, by $\BW/\bcU_k$ the quotient bundle, and by $\bcU_k^\perp$ its dual.
The point of the Grassmannian corresponding to a subspace $U_k \subset W$ is denoted by $[U_k]$, or even just $U_k$.
We recall that $\det \cU_k \cong \det \cU_k^\perp \cong \cO(-1)$ is the antiample generator of $\Pic(\Gr(k,W))$, and analogously for $\Gr(k,\BW)$.

For a vector bundle $\cV$ on a scheme $S$ we denote by $\Gr_S(k,\cV)$ the relative Grassmannian, parameterizing $k$-dimensional subspaces in the fibers of $\cV$.
In particular, we consider the two-step flag variety
\begin{equation*}
\Fl(k_1,k_2;V) \cong \Gr_{\Gr(k_2,V)}(k_1,\cU_{k_2}) \cong \Gr_{\Gr(k_1,V)}(k_2-k_1,V/\cU_{k_1}).
\end{equation*}
We denote by $\cU_{k_1} \hookrightarrow \cU_{k_2} \hookrightarrow V \otimes \cO$ the tautological flag of subbundles on $\Fl(k_1,k_2;V)$.
In particular, we abuse the notation by using the same name for the tautological vector bundle on the Grassmannian and its pullback to the flag variety.


Given a morphism $\varphi:\cE \to \cF$ of vector bundles on a scheme $S$ we denote by $D_k(\varphi) \subset S$ its $k$-th degeneration scheme,
i.e.\ the subscheme of $S$ whose ideal is locally generated by all $(r+1-k)\times (r+1-k)$ minors of the matrix of $\varphi$ for $r = \min\{\rank(\cE),\rank(\cF)\}$.
This is a closed subscheme in $S$, and $D_{k+1}(\varphi) \subset D_k(\varphi)$.

\subsection{A blowup Lemma}\label{subsection:blowup-lemma}

We will use several times the following observation, which is quite classical. Unfortunately, we were not able to find a reference for it, so we sketch a short proof.

\begin{lemma}\label{lemma:blowup}
Let $\varphi:\cE \to \cF$ be a morphism of vector bundles of ranks $\rk(\cF) = r$ and $\rk(\cE) = r+1$ on a Cohen--Macaulay scheme $S$.
Denote by $D_k(\varphi)$ the $k$-th degeneracy locus of $\varphi$.
Consider the projectivization $p:\P_S(\cE) \to S$, then $\varphi$ gives a global section of the vector bundle $p^*\cF \otimes \cO(1)$.
If $\codim D_k(\varphi) \ge k + 1$ for all $k \ge 1$ then the zero locus of $\varphi$ on $\P_S(\cE)$ is isomorphic to the blowup of $S$ with center in the degeneration locus $D_1(\varphi)$.

Moreover, in this case the line bundle corresponding to the exceptional divisor of the blowup is isomorphic to $\det(\cE^\vee) \otimes \det(\cF) \otimes \cO(-1)$.
Finally, if the second degeneracy locus~$D_2(\varphi)$ is empty, then the exceptional divisor is isomorphic 
to the projectivization of the vector bundle $\Ker(\varphi\vert_{D_1(\varphi)}:\cE\vert_{D_1(\varphi)} \to \cF\vert_{D_1(\varphi)})$.
\end{lemma}
\begin{proof}
Consider the morphism $\varphi^\vee:\cF^\vee \to \cE^\vee$. 
By assumption, it is generically injective, hence is a monomorphism of sheaves.
Let $\cC := \Coker(\varphi^\vee)$ be its cokernel. 
Let us show it is torsion free.
Indeed, dualizing the sequence $0 \to \cF^\vee \to \cE^\vee \to \cC \to 0$ we obtain
\begin{equation*}
0 \to \cC^\vee \to \cE \xrightarrow{\ \varphi\ } \cF \to \operatorname{\mathscr{E}\!\mathit{xt}}^1(\cC,\cO_S) \to 0.
\end{equation*}
The sheaf $\operatorname{\mathscr{E}\!\mathit{xt}}^1(\cC,\cO_S)$ is supported on the degeneracy locus $D_1(\varphi)$, hence in codimension 2.
Therefore, by Cohen--Macaulay property, 
we have $\operatorname{\mathscr{E}\!\mathit{xt}}^{\le 1}(\operatorname{\mathscr{E}\!\mathit{xt}}^1(\cC,\cO_S),\cO_S) = 0$, 
so dualizing the above sequence we get an exact sequence
\begin{equation*}
0 \to \cF^\vee \xrightarrow{\ \varphi^\vee\ } \cE^\vee \to \cC^{\vee\vee} \to \operatorname{\mathscr{E}\!\mathit{xt}}^{2}(\operatorname{\mathscr{E}\!\mathit{xt}}^1(\cC,\cO_S),\cO_S).
\end{equation*}
Comparing it with the definition of the sheaf $\cC$, we conclude that $\cC$ embeds into the torsion free sheaf~$\cC^{\vee\vee}$, hence is itself torsion free.

Next, consider the morphism
\begin{equation*}
\cE^\vee \cong \det(\cE^\vee) \otimes \bw{r}\cE \xrightarrow{\ \wedge^r\varphi\ } \det(\cE^\vee) \otimes \bw{r}\cF \cong \det(\cE^\vee) \otimes \det(\cF).
\end{equation*}
Its composition with $\varphi^\vee$ is zero, hence it induces a morphism $\cC \to \det(\cE^\vee) \otimes \det(\cF)$.
This morphism is an isomorphism away of $D_1(\varphi)$, so since $\cC$ is torsion free, it is injective and we have a left exact sequence
\begin{equation*}
0 \to \cF^\vee \xrightarrow{\ \varphi^\vee\ } \cE^\vee \xrightarrow{\ \wedge^r\varphi\ } \det(\cE^\vee) \otimes \det(\cF).
\end{equation*}
But since the scheme structure on $D_1(\varphi)$ is given by the minors of size $r$ of $\varphi$, i.e.\ by the entries of $\wedge^r\varphi$,
the image of $\wedge^r\varphi$ is the twist of the ideal of $D_1(\varphi)$, hence we have an exact sequence
%
%
\begin{equation}\label{eq:exact-seq-varphi}
0 \to \cF^\vee \xrightarrow{\ \varphi^\vee\ } \cE^\vee \xrightarrow{\ \wedge^r\varphi\ }  \det(\cE^\vee) \otimes \det(\cF) \to \det(\cE^\vee) \otimes \det(\cF)\vert_{D_1(\varphi)} \to 0.
\end{equation}

Let $\tilde{S} \subset \P_S(\cE)$ be the zero locus of $\varphi$ on $\P_S(\cE)$. 
Since $\tilde{S}$ is the zero locus of a rank~$r$ vector bundle on a Cohen--Macaulay variety $\P_S(\cE)$ of dimension $\dim(S) + r$, the dimension of any component of $\tilde{S}$ is greater or equal than~$\dim S$.
On the other hand, the fibers of $\tilde{S}$ over $D_k(\varphi) \setminus D_{k+1}(\varphi)$ are projective spaces of dimension $k$, 
hence $\tilde{S}$ has a stratification with strata of dimension $\dim(D_k(\varphi)) + k$ which is less than $\dim S$ for $k \ge 1$.
It follows that $\tilde{S}$ is irreducible of dimension $\dim\tilde{S} = \dim S$, and the map $p:\tilde{S} \to S$ is birational.

Now let us compute the pushforward to $S$ of the relatively very ample line bundle $\cO(1)\vert_{\tilde{S}}$.
As we already have seen, the dimension of the zero locus of the section $\varphi$ of $p^*\cF \otimes \cO(1)$ on~$\P_S(\cE)$ equals $\dim \tilde{S} = \dim \P_S(\cE) - r$, hence the section is regular, so the Koszul complex
\begin{equation*}
0 \to \bw{r}(p^*\cF^\vee) \otimes \cO(-r) \to \dots \to \bw2(p^*\cF^\vee) \otimes \cO(-2) \to p^*\cF^\vee \otimes \cO(-1) \to \cO \to \cO_{\tilde{S}} \to 0
\end{equation*}
is a resolution of the structure sheaf of $\tilde{S}$.
Twisting it by $\cO(1)$ and pushing forward to $S$, we obtain an exact sequence
\begin{equation*}
0 \to \cF^\vee \xrightarrow{\ \varphi^\vee\ } \cE^\vee \to p_*(\cO_{\tilde{S}}(1)) \to 0.
\end{equation*}
Comparing it with~\eqref{eq:exact-seq-varphi} we deduce an isomorphism with a twisted ideal of $D_1(\varphi)$:
\begin{equation*}
p_*(\cO_{\tilde{S}}(1)) \cong \cI_{D_1(\varphi)} \otimes \det(\cE^\vee) \otimes \det(\cF).
\end{equation*}
It follows that $p:\tilde{S} \to S$ is the blowup of the ideal $\cI_{D_1(\varphi)}$.

Moreover, it follows also that the pushforward of $\det(\cE) \otimes \det(\cF^\vee) \otimes \cO(1)$ is isomorphic to the ideal $\cI_{D_1(\varphi)}$, 
hence this line bundle corresponds to the minus exceptional divisor of the blowup. 
This proves the second part of the Lemma. 
The last part of the Lemma is evident.
\end{proof}

\section{Geometry of general skew forms and generality assumptions}\label{section:forms-and-generality}

We say that a skew-symmetric $p$-form on a vector space $V$ of dimension $n$ is {\sf general}, if its $\PGL(V)$-orbit in $\P(\bw{p}V^\vee)$ is open.
It is a classical fact, that a 2-form is general if and only if its rank is equal to $2\lfloor n/2 \rfloor$,
and that a general 3-form exists if and only if $n \le 8$.

In this section we remind a description of general 3-forms on vector spaces of dimensions 6 and 7 (we will not need the 8-dimensional case so we skip it here, however an interested reader can find a discussion of these in~\cite{kuznetsov2015kuchle}). 
We also discuss some natural subschemes of Grassmannians associated with these forms.

After that we pass to the situation, which is the most important for the rest of the paper: a pair $(\lambda,\mu)$ consisting
of a 4-form and a 2-form on a vector space of dimension 7. We discuss how this pair looks under some generality assumptions
and give an explicit standard presentation for such a pair.


\subsection{A 3-form on a 6-space}\label{subsection:3form-6space}

Let $\BW$ be a vector space of dimension 6.
The following description of general 3-forms on $\BW$ is well known.

\begin{lemma}\label{lemma:3form-general}
A $3$-form $\blam \in \bw3\BW^\vee$ is general if and only if there is a direct sum decomposition
\begin{equation}\label{eq:bw-decomposition}
\BW = A_1 \oplus A_2,
\qquad\qquad 
\dim A_1 = \dim A_2 = 3,
\end{equation}
and 
\begin{equation}\label{eq:blam-general}
\blam = \blam_1 + \blam_2
\end{equation}
for decomposable $3$-forms $0 \ne \blam_1 \in \bw3A_1^\perp$ and $0 \ne \blam_2 \in \bw3A_2^\perp$.
\end{lemma}

In appropriate coordinates a general 3-form can be written as 
\begin{equation}\label{eq:blam-explicit}
\blam = x_{123} + x_{456}.
\end{equation} 
In the rest of the paper we always denote by $A_1$ and $A_2$ the summands of the canonical direct sum decomposition of $\BW$ associated with a general 3-form $\blam$.


\begin{lemma}\label{lemma:gr26-3form}
Assume $\blam$ is a general $3$-form on $\BW$.
Consider the Grassmannian $\Gr(2,\BW)$ and let $\bcU_2 \subset \BW \otimes \cO$ be the tautological subbundle of rank $2$.
The zero locus of the global section~$\blam$ of the vector bundle $\bcU_2^\perp(1)$ 
is isomorphic to the product $\P(A_1) \times \P(A_2)$, and moreover
\begin{equation}\label{eq:gr26-3form-bu2}
\begin{aligned}
\bcU_2\vert_{\P(A_1) \times \P(A_2)} &\cong \cO(-h_1) \oplus \cO(-h_2),
\\
\bcU_2^\perp\vert_{\P(A_1) \times \P(A_2)} & \cong \Omega_{\P(A_1)}(h_1) \oplus \Omega_{\P(A_2)}(h_2),
\end{aligned}
\end{equation}
where $h_1$ and $h_2$ are the hyperplane classes on $\P(A_1)$ and $\P(A_2)$ respectively.
\end{lemma}
\begin{proof}
The direct sum decomposition~\eqref{eq:bw-decomposition} induces a decomposition
\begin{equation*}
\bw2\BW = \bw2A_1 \oplus \bw2A_2 \oplus (A_1 \otimes A_2).
\end{equation*}
The map $\bw2\BW \xrightarrow{\ \blam\ } \BW^\vee = A_1^\vee \oplus A_2^\vee$ is zero on the third summand 
and gives isomorphisms of the first two summands with $A_1^\vee$ and $A_2^\vee$ respectively.
Hence the zero locus of $\blam$ on $\Gr(2,\BW)$ is the intersection $\Gr(2,\BW) \cap \P(A_1 \otimes A_2) \subset \P(\bw2\BW)$.
It remains to note that the restrictions of the Pl\"ucker quadrics from $\P(\bw2\BW)$ to $\P(A_1 \otimes A_2)$
are the Segre quadrics, cutting out $\P(A_1) \times \P(A_2) \subset \P(A_1 \otimes A_2)$. 

Geometrically, this means that a 2-subspace $\BU_2 \subset \BW$ is in the zero locus of $\blam$ if and only if 
it intersects both subspaces $A_1$ and $A_2$. This means that the restriction of the tautological bundle to $\P(A_1)\times\P(A_2)$
is the direct sum of the pullbacks of the tautological line bundles on $\P(A_1)$ and $\P(A_2)$, i.e.\ gives the first part of~\eqref{eq:gr26-3form-bu2}.
Further, it follows that for the quotient bundle we have
\begin{equation*}
(\BW/\bcU_2)\vert_{\P(A_1) \times  \P(A_2)} \cong
(A_1 \otimes \cO)/\cO(-h_1) \oplus (A_2 \otimes \cO)/\cO(-h_2),
\end{equation*}
so it is isomorphic to the direct sum of pullbacks of the twisted tangent bundles. Dualizing, we get the second part of~\eqref{eq:gr26-3form-bu2}.
%
%
%
%
%
\end{proof}

\begin{corollary}\label{corollary:gr46-3form}
Assume $\blam$ is a general $3$-form on $\BW$.
Consider the Grassmannian $\Gr(4,\BW)$ and let $\bcU_4 \subset \BW \otimes \cO$ be the tautological subbundle of rank $4$.
The zero locus of the global section $\blam$ of the vector bundle $\bw3\bcU_4^\vee$ 
is isomorphic to the product $\Gr(2,A_1) \times \Gr(2,A_2)$, and moreover
\begin{equation}\label{eq:gr46-3form-bu4}
\bcU_4\vert_{\Gr(2,A_1) \times \Gr(2,A_2)} \cong \cU_{2,A_1} \oplus \cU_{2,A_2},
\end{equation}
where $\cU_{2,A_1}$ and $\cU_{2,A_2}$ are the tautological bundles on $\Gr(2,A_1)$ and $\Gr(2,A_2)$ respectively.
\end{corollary}
\begin{proof}
We have a canonical isomorphism $\Gr(4,\BW) \cong \Gr(2,\BW^\vee)$, that takes the bundle $\bw3\bcU_4^\vee$ on the first Grassmannian to the bundle $(\bcU'_2)^\perp(1)$ on the second, 
and the global section $\blam$ of the first to the global section $\blam^\vee$ of the second. 
Hence the zero locus of $\blam$ in $\Gr(4,\BW)$ is isomorphic to the global section of $\blam^\vee$ in $\Gr(2,\BW^\vee)$.
Clearly, $\blam^\vee$ is a general 3-form on $\BW^\vee$ corresponding to the direct sum decomposition $\BW^\vee = A_1^\vee \oplus A_2^\vee$.
Hence by Lemma~\ref{lemma:gr26-3form} the zero locus is isomorphic to $\P(A_1^\vee) \times \P(A_2^\vee) \cong  \Gr(2,A_1) \times \Gr(2,A_2)$.
The statement about the restriction of the tautological bundle follows from the second part of~\eqref{eq:gr26-3form-bu2}.
\end{proof}


Every 2-form $\mu \in \bw2\BW^\vee$ induces a pairing between the subspaces $A_1,A_2 \subset \BW$.

\begin{lemma}\label{lemma:gr26-3form-2form}
Assume $\blam$ is a general $3$-form on $\BW$.
If the pairing between the subspaces $A_1$ and~$A_2$ induced by a $2$-form $\mu$ is nondegenerate then
the zero locus of the global section $(\blam,\mu)$ of the vector bundle $\bcU_2^\perp(1) \oplus \cO(1)$ on the Grassmannian $\Gr(2,\BW)$
is isomorphic to the flag variety $\Fl(1,2;A_1) \cong \Fl(1,2;A_2) \subset \P(A_1) \times \P(A_2)$.
\end{lemma}
\begin{proof}
By Lemma~\ref{lemma:gr26-3form} the zero locus of $\blam$ is isomorphic to $\P(A_1) \times \P(A_2)$. By~\eqref{eq:gr26-3form-bu2} we have
\begin{equation*}
\cO(1)\vert_{\P(A_1)\times \P(A_2)} \cong
\det(\bcU_2^\vee)\vert_{\P(A_1)\times \P(A_2)} \cong
\cO(h_1 + h_2),
\end{equation*}
hence the zero locus of $\mu$ in ${\P(A_1)\times \P(A_2)}$ is a divisor of bidegree $(1,1)$.
Clearly, this divisor corresponds to the pairing between $A_1$ and $A_2$ induced by the form $\mu$.
So, if this pairing is nondegenerate, it identifies $A_2$ with $A_1^\vee$, and the corresponding divisor
with the flag variety.
\end{proof}

\subsection{A 4-form on a 7-space}\label{subsection:4form-4space}

Let $W$ be a vector space of dimension 7. 
Recall that with each $4$-form $\lambda$ on $W$ we can associate a 3-form $\lambda^\vee$ on the dual space.
Further on we will work more with 4-forms, but some geometric constructions are better adapted to 3-forms, so it is useful to keep this correspondence in mind.

Consider a 3-form $\lambda^\vee \in \bw3W$ on $W^\vee$ as a global section
of the vector bundle $\Omega^2_{\P(W^\vee)}(3)$. It gives a morphism of vector bundles
\begin{equation*}
T_{\P(W^\vee)} \xrightarrow{\ \lambda^\vee\ } \Omega_{\P(W^\vee)}(3),
\end{equation*}
which is easily seen to be skew-symmetric (up to a twist). Since 
\begin{equation*}
\det(\Omega_{\P(W^\vee)}(3)) \otimes \det(T_{\P(W)})^{-1} \cong 
\cO(-7 + 3\cdot 6 - 7) = 
\cO(4),
\end{equation*}
the Pfaffian of the map $\lambda^\vee$ is a section of $\cO(2)$, hence the degeneracy degeneracy locus of $\lambda^\vee$ is either a quadric, or the whole space. 
We denote this degeneracy locus by $\bQ^\vee_\lambda \subset \P(W^\vee)$ (geometrically, $\bQ^\vee_\lambda$ is the set of vectors $w^\vee \in W^\vee$ such that the rank of the 2-form $\lambda^\vee \conv w^\vee$ is less than 6) 
and by $\bq_\lambda \in \Sym^2 W$ its equation.
In case $\bQ^\vee_\lambda$ is a smooth quadric, we denote by~$\bQ_\lambda \subset \P(W)$ its projective dual, it is also smooth then and its equation is $\bq^{-1}_\lambda \in \Sym^2W^\vee$.

\begin{lemma}\label{lemma:4form-general}
The following conditions are equivalent:
\begin{enumerate}\renewcommand{\theenumi}{\roman{enumi}}
\item 
a $4$-form $\lambda \in \bw4W^\vee$ is general;
\item 
a $3$-form $\lambda^\vee \in \bw3W$ is general;
\item 
the degeneracy locus $\bQ^\vee_\lambda \subset \P(W^\vee)$ is a smooth quadric;
\item 
$\lambda = x_{0123} + x_{0456} + x_{1256} + x_{1346} + x_{2345}$ in some basis;
\item 
$\lambda^\vee = e_{456} + e_{123} + e_{034} + e_{025} + e_{016}$ in some basis.
\end{enumerate}
\end{lemma}

A straightforward computation shows that for $\lambda = x_{0123} + x_{0456} + x_{1256} + x_{1346} + x_{2345}$, the associated quadric $\bQ_\lambda$ is
\begin{equation}\label{eq:qlambda-explicit}
\bQ_\lambda = \{ x_0^2 - x_1x_6 - x_2x_5 - x_3x_4 = 0 \} \subset \P(W).
\end{equation} 

The stabilizer of a general 4-form $\lambda$ in $\GL(W)$ is the simple algebraic group $\Gtwo$, the space $W$ is one of its fundamental representations, and the action of $\Gtwo$ on $\P(W)$ has just two orbits. 
The closed orbit is the quadric $\bQ_\lambda$, and the open orbit is its complement.
In particular, $\bQ_\lambda$ is one of the two minimal compact homogeneous spaces of $\Gtwo$.


\begin{lemma}\label{lemma:la-w-gen}
Let $\lambda \in \bw4W^\vee$ be a general $4$-form.
For every vector $w \in W$ such that $\bq^{-1}_\lambda(w,w) = 1$, let $w^\vee = \bq^{-1}_\lambda(w) \in W^\vee$ be its polar covector. 
Then there is a direct sum decomposition of the space $W$ and the corresponding decomposition of the form $\lambda$
\begin{equation*}
W = \k w \oplus \Ker(w^\vee),
\qquad
\lambda = w^\vee \wedge \blam + \lambda',
\end{equation*}
with $\blam \in \bw3(\Ker(w^\vee))^\vee$ and $\lambda' \in \bw4(\Ker(w^\vee))^\vee$.
Moreover, $\blam$ is a general $3$-form on $\Ker(w^\vee)$, and if $\Ker(w^\vee) = A_1 \oplus A_2$
is the corresponding direct sum decomposition, then the form $\lambda'$ annihilates $\bw3A_1$ and $\bw3A_2$,
and defines a nondegenerate pairing between $\bw2A_1$ and $\bw2A_2$.
%
%
%
\end{lemma}
\begin{proof}
Since the group $\Gtwo$ acts transitively on the complement of $\bQ_\lambda$, it is enough to check all the properties for just one vector $w$. 
So, choose a basis as in Lemma~\ref{lemma:4form-general}(iv), so that the quadric $\bQ_\lambda$ is given by~\eqref{eq:qlambda-explicit},
and take $w = e_0$. Then $w^\vee  = x_0$, and hence
\begin{equation*}
\blam = x_{123} + x_{456}
\qquad \text{and}\qquad
\lambda' = x_{1256} + x_{1346} + x_{2345}.
\end{equation*}
Thus $\blam$ corresponds to the direct sum decomposition of $\BW = A_1 \oplus A_2$ with $A_1 = \langle e_1,e_2,e_3 \rangle$ and $A_2 = \langle e_4, e_5, e_6 \rangle$, 
the form $\lambda'$ annihilates both $A_1$ and $A_2$ (i.e.\ $\lambda' \conv e_{123} = \lambda' \conv e_{456} = 0$),
and the pairing between the spaces $\bw2A_1 = \langle e_{12}, e_{13}, e_{23} \rangle$ and $\bw2A_2 = \langle e_{45}, e_{46}, e_{56} \rangle$ given by $\lambda'$ is nondegenerate.
%
%
%
%
%
\end{proof}

The other minimal compact homogeneous variety of $\Gtwo$ can be described as follows. 

%


\begin{lemma}[\cite{mukai1989fano}]\label{lemma:g2-grassmannian}
The zero locus of the global section $\lambda \in H^0(\Gr(5,W),\bw4\cU_5^\vee)$ in $\Gr(5,W)$
is a minimal compact homogeneous variety of the group $\Gtwo$.
\end{lemma}

This homogeneous variety is usually called the {\sf adjoint variety} of group $\Gtwo$.
In what follows we will denote it by $\Gr_\lambda(5,W)$.

By duality, one can get another description of the adjoint variety.
Recall the canonical isomorphism $\Gr(5,W) \cong \Gr(2,W^\vee)$ (defined by $(U_5 \subset W) \mapsto (U_5^\perp \subset W^\vee)$).
By~\eqref{eq:convolution-duality} this isomorphism takes $\Gr_\lambda(5,W)$ to the subvariety of $\Gr(2,W^\vee)$ parameterizing all 2-subspaces $U'_2 \subset W^\vee$ annihilated by the 3-form $\lambda^\vee$ on $W^\vee$. 
We denote this subvariety by $\Gr_{\lambda^\vee}(2,W^\vee)$.

\begin{remark}\label{remark:gr27-g2}
Note that from this description it is clear that if $w^\vee \in \P(W^\vee) \setminus \bQ^\vee_\lambda$, then $w^\vee$ is not contained in any 2-subspace $U'_2 \subset W^\vee$, corresponding to a point of $\Gr_{\lambda^\vee}(2,W^\vee)$.
\end{remark}




\subsection{The structure of the data and genericity assumptions}\label{subsection:genericity-assumptions}

Let $W$ be a vector space of dimension 7, and
\begin{equation*}
\lambda \in \bw4W^\vee,
\qquad
\mu \in \bw2W^\vee,
\end{equation*}
be a 4-form and a 2-form on $W$. In this section we discuss how a pair $(\lambda,\mu)$ looks under some genericity assumptions 
and introduce some notions that will be actively used further on.

\underline{Assumption 1:} We assume that the form $\lambda$ is general (i.e.\ lies in the open $\PGL(W)$-orbit), the form $\mu$ is general (i.e.\ lies in the open $\PGL(W)$-orbit), 
and the kernel of the form $\mu$ is in a general position with respect to the form $\lambda$.

As it was discussed above, these assumptions can be explicitly reformulated as
\begin{equation}\label{assumption:lambda-mu-general}
\begin{cases}
\text{the quadrics $\bQ_\lambda^\vee \subset \P(W^\vee)$ and $\bQ_\lambda \subset \P(W)$ are smooth;}\\
\rank(\mu) = 6;\\
w_0 \not\in \bQ_\lambda \subset \PP(W).
\end{cases}
\end{equation}
Here $\bQ_\lambda^\vee$ and $\bQ_\lambda$ are the quadrics defined in section~\ref{subsection:4form-4space}, and $w_0$ is defined to be a generator of the one-dimensional kernel space of $\mu$ on $W$,
normalized by the condition $\bq_\lambda^{-1}(w_0,w_0) = 1$.

%
%
%

As in Lemma~\ref{lemma:la-w-gen} we denote by $w_0^\vee = \bq^{-1}_\lambda(w_0)$ the polar covector to the point $w_0$, so that
\begin{equation*}\label{eq:w0v-w0}
w_0^\vee(w_0) = 1,
\end{equation*} 
and consider the direct sum decomposition defined by the pair $(w_0,w_0^\vee)$:
\begin{equation}\label{eq:bw}
W = \k w_0 \oplus \BW,
\qquad
\BW = \Ker w_0^\vee \subset W,
\end{equation} 
and the induced decomposition of the form $\lambda$ 
\begin{equation}\label{eq:la-la-la}
\lambda = w_0^\vee \wedge \blam + \lambda',
\qquad\qquad
\text{with $\blam \in \bw3\BW^\vee$ and $\lambda' \in \bw4\BW^\vee$.}
\end{equation}
Since $w_0$ generates the kernel of $\mu$, the form $\mu$ can be considered just as a form in $\bw2\BW^\vee$.

Thus, under assumption~\eqref{assumption:lambda-mu-general}, the data $(\lm)$ reduces to the data $(\lambda',\blam,\mu)$ of a 4-form, a 3-form, and a 2-form on a 6-dimensional vector space $\BW$,
such that the 2-form $\mu$ is non-degenerate, the 3-form $\blam$ is general and corresponds to a decomposition $\BW = A_1 \oplus A_2$,
and the 4-form $\lambda'$ annihilates the subspaces $A_1$ and $A_2$, and induces a nondegenerate pairing between the spaces $\bw2A_1$ and $\bw2A_2$ (see Lemma~\ref{lemma:la-w-gen}).

The 2-form $\mu \in \bw2\BW^\vee$ gives a 4-form $\mu^2 := \mu \wedge \mu \in \bw4\BW^\vee$. 
We consider the pencil
\begin{equation*}
t\lambda' + \mu^2 \in \bw4\BW^\vee,
\qquad
t \in \k
\end{equation*}
of 4-forms on $\BW$ generated by $\lambda'$ and $\mu^2$. 
Both forms give a pairing between the spaces $\bw2A_1$ and $\bw2A_2$.
Thinking of these pairings as of linear maps $\bw2A_1 \to \bw2A_2^\vee$, that linearly depend on $t$,
we have a cubic polynomial 
\begin{equation*}
\chi_{A_1,A_2}^{\lambda',\mu^2}(t) := \det\Big(t\lambda' + \mu^2 \colon \bw2A_1 \to \bw2A_2^\vee\Big) \in \k[t].
\end{equation*}

\underline{Assumption 2:}
We assume that the polynomial $\chi_{A_1,A_2}^{\lambda',\mu^2}$ is general. 
Explicitly, we assume 
\begin{equation}\label{assumption:pencil-regular}
\text{the polynomial $\chi_{A_1,A_2}^{\lambda',\mu^2}$ has three distinct roots.}
\end{equation}
This assumption allows to make the form of the data fairly explicit.

\begin{lemma}\label{lemma:bases-a1-a2}
Assume a $3$-form $\blam$ on $\BW$ is general and corresponds to a direct sum decomposition $\BW = A_1 \oplus A_2$, 
and the form $\lambda'$ annihilates $A_1$ and $A_2$ and defines a nondegenerate pairing between $\bw2A_1$ and $\bw2A_2$.
If~\eqref{assumption:pencil-regular} is satisfied then
%
one can choose a basis $e_1,e_2,e_3$ in~$A_1$ and a basis $e_4,e_5,e_6$ in $A_2$ such that 
\begin{equation}\label{eq:lambda-mu-explicit}
\begin{aligned}
\blam & = x_{123} + x_{456},\\
\lambda' & = x_{1256} + x_{1346} + x_{2345},\\
\mu^2 & = (M_1x_{1456} + M_2x_{2456} + M_3x_{3456}) + (M_4x_{1234} + M_5x_{1235} + M_6x_{1236}) \\
& + (K_1x_{2345} + K_2x_{1346} + K_3x_{1256}),
\end{aligned}
\end{equation}
where $M_i, K_i \in \k$, and $K_i$ are pairwise distinct.
Such basis is unique up to rescaling and permuting.
%
%
\end{lemma}
\begin{proof}
Since the pairing defined by the form $\lambda'$ is nondegenerate, it gives an isomorphism $\lambda':\bw2A_1 \isomto \bw2A_2^\vee$, so we can use it to identify the two spaces.
Then $\mu^2$ becomes an endomorphism of $\bw2A_1$. 
Condition~\eqref{assumption:pencil-regular} means that the characteristic polynomial of this endomorphism has three distinct roots, hence the endomorphism can be diagonalized in an appropriate basis. 
Choosing such a basis in $\bw2A_1 \cong A_1^\vee$, considering its dual basis in $A_1$, and transferring the basis to $\bw2A_2^\vee \cong A_2$ via $\lambda'$,
we obtain bases in $A_1$ and $A_2$ such that $\lambda'$ has the required form and the pairing given by $\mu^2$ between $\bw2A_1$ and $\bw2A_2$ is diagonal. 
Rescaling the bases we can also ensure that $\blam$ has the required form. Also it is clear that the bases defined in this way are unique up to rescaling and permuting.
%
\end{proof}


If the 2-form $\mu$ is nondegenerate, then the 4-form $\mu^2$, considered as a map $\BW^\vee \to \BW$, is inverse (up to a scalar factor) to~$\mu$, considered as a map $\BW \to \BW^\vee$.
Thus, $\mu$ can be reconstructed from $\mu^2$. A straightforward computation shows that 
%
$\mu$ is proportional to the skew-symmetric matrix with the following upper-triangular part
\begin{equation}\label{eq:mu-matrix}
\mu = 
\left(\begin{array}{rrr|rrr}
0 & M_4K_3 & -M_5K_2 & 	M_1M_4 & M_1M_5 & K_2K_3 + M_1M_6 \\
& 0 & M_6K_1 &		M_2M_4 & K_1K_3 + M_2M_5 & M_2M_6 \\
&& 0 &			K_1K_2 + M_3M_4 & M_3M_5 & M_3M_6 \\
\hline
&&& 0 & M_1K_1 & -M_2K_2 \\
&&&& 0 & M_3K_3 \\
&&&&& 0
\end{array}\right)
\end{equation}

\underline{Assumption 3:}
We assume that all the coefficients of the matrix $\mu^2$ in~\eqref{eq:lambda-mu-explicit} are nonzero:
%
\begin{equation}\label{assumption:m-nonzero}
\text{$M_i \ne 0$ for all $1 \le i \le 6$ and $K_i \ne 0$ for all $1 \le i \le 3$.}
\end{equation} 

We also have some nonvanishings for free.

\begin{lemma}
If the skew form $\mu$ is nondegenerate then 
\begin{equation}\label{eq:mmk-nonzero}
M_1M_6K_1 + M_2M_5K_2 + M_3M_4K_3 + K_1K_2K_3 \ne 0.
\end{equation}
If, moreover, \eqref{assumption:m-nonzero} holds, then the pairing between $A_1$ and $A_2$ defined by $\mu$ is nondegenerate.
\end{lemma}
\begin{proof}
A straightforward computation shows that the Pfaffian of the matrix~\eqref{eq:mu-matrix} is equal to
\begin{equation*}
\Pf(\mu) = (M_1M_6K_1 + M_2M_5K_2 + M_3M_4K_3 + K_1K_2K_3 )^2.
\end{equation*}
In particular, the nondegeneracy of $\mu$ implies~\eqref{eq:mmk-nonzero}.
Furthermore, a computation of the determinant of the upper-right 3-by-3 block $\mu_{A_1,A_2}$ in~\eqref{eq:mu-matrix} gives
\begin{equation*}
\det(\mu_{A_1,A_2}) = -K_1K_2K_3(M_1M_6K_1 + M_2M_5K_2 + M_3M_4K_3 + K_1K_2K_3),
\end{equation*}
so assuming that all $K_i$ are nonzero, we deduce that the pairing is nondegenerate.
\end{proof}

%

Later we will need the following consequence of assumption~\eqref{assumption:m-nonzero}.

\begin{lemma}\label{lemma:no21cases}
Assume $\lambda'$ and $\mu$ are given by~\eqref{eq:lambda-mu-explicit} and~\eqref{eq:mu-matrix} with $K_i$ pairwise distinct, and assume that~\eqref{assumption:m-nonzero} holds.

\noindent$(a)$
If for a two-dimensional subspace $U_{2,A_1} \subset A_1$ and a one-dimensional subspace $U_{1,A_2} \subset A_2$ the subspace $U_{2,A_1} \oplus U_{1,A_2} \subset \BW$ is $\mu$-isotropic, then $\lambda'$ does not annihilate it.

\noindent$(b)$
If for a one-dimensional subspace $U_{1,A_1} \subset A_1$ and a two-dimensional subspace $U_{2,A_2} \subset A_2$ the subspace $U_{1,A_1} \oplus U_{2,A_2} \subset \BW$ is $\mu$-isotropic, then $\lambda'$ does not annihilate it.
%
\end{lemma}
\begin{proof}
Let $\alpha_1 = s_1e_{23} - s_2e_{13} + s_3e_{12}$ be the bivector corresponding to the subspace $U_{2,A_1}$ and $a_2 = s_4e_4 + s_5e_5  +s_6e_6$ be the vector corresponding to $U_{1,A_2}$. 
Assume that the subspace $U_{2,A_1} \oplus U_{1,A_2}$ is annihilated by $\lambda'$ and is $\mu$-isotropic. 
It is easy to see that then it is also annihilated by $\mu^2$ (see the proof of Lemma~\ref{lemma:lambda-mu-shift} below). 
Then
\begin{align*}
\lambda' \conv (\alpha_1 \wedge a_2) & = (s_1s_5 - s_2s_6)x_4 + (-s_1s_4 + s_3s_6)x_5 + (s_2s_4 - s_3s_5)x_6 && = 0,\\
\mu^2 \conv (\alpha_1 \wedge a_2) & = (K_1s_1s_5 - K_2s_2s_6)x_4 + (-K_1s_1s_4 + K_3s_3s_6)x_5 + (K_2s_2s_4 - K_3s_3s_5)x_6 && = 0.
\end{align*}
Assume $s_4 \ne 0$. 
Subtracting the second line from the first line multiplied by $K_3$ and considering the coefficients at $x_5$ and~$x_6$ gives $(K_1-K_3)s_1s_4 = (K_2 - K_3)s_2s_4 = 0$.
Since $K_i$ are pairwise distinct, it follows that $s_1 = s_2 = 0$, hence $\alpha_1 = s_3e_{12}$ and $\mu \conv \alpha_1 = M_4K_3s_3 \ne 0$.
Analogously, if $s_5 \ne 0$ we conclude that $\alpha_1 = -s_2e_{13}$ and $\mu \conv \alpha_1 = M_5K_2s_2 \ne 0$, and if $s_6 \ne 0$ we conclude that $\alpha_1 = s_1e_{23}$ and $\mu \conv \alpha_1 = M_6K_1s_1 \ne 0$.
The second statement is proved similarly.
\end{proof}

\underline{Assumption 4:}
Consider the Lagrangian Grassmannian $\LGr_\mu(3,\BW)$
corresponding to the symplectic form $\mu$ on $\BW$. By definition
$\LGr_\mu(3,\BW) \subset \Gr(3,\BW) \subset \P(\bw3\BW)$,
hence the 3-form $\blam \in \bw3\BW^\vee$ gives its hyperplane section. We assume
\begin{equation}\label{assumption:lgr-hyperplane}
\text{the hyperplane section $\LGr_\mu(3,\BW)_\blam$ of $\LGr_\mu(3,\BW)$ given by $\blam$ is smooth.}
\end{equation}
In fact, it is possible that~\eqref{assumption:lgr-hyperplane} follows from the other assumptions we have made, but this seems to be not so easy to prove.

Now we want to summarize the results of this section.

\begin{proposition}
For a triple $(\lambda',\blam,\mu)$ consisting of a $4$-form, a $3$-form, and a $2$-form on $\BW$ such that assumptions~\eqref{assumption:lambda-mu-general} and~\eqref{assumption:pencil-regular} hold,
there is a basis in $\BW$ such that the forms are given by~\eqref{eq:lambda-mu-explicit} and~\eqref{eq:mu-matrix} with $K_i$ pairwise distinct and~\eqref{eq:mmk-nonzero} satisfied. 

Conversely, if the forms are given by~\eqref{eq:lambda-mu-explicit} and~\eqref{eq:mu-matrix} with pairwise distinct $K_i$ and~\eqref{eq:mmk-nonzero} holds, 
then assumptions~\eqref{assumption:lambda-mu-general} and~\eqref{assumption:pencil-regular} hold.
If, moreover, \eqref{assumption:m-nonzero} holds, then the pairing between the subspaces $A_1$ and $A_2$ defined by $\mu$ is nondegenerate and the assertion of Lemma~{\rm\ref{lemma:no21cases}} holds.

Finally, for general $M_i$ and $K_i$ assumption~\eqref{assumption:lgr-hyperplane} holds.
%
%
%
\end{proposition}
\begin{proof}
This is clear since all the assumptions are open conditions. 
\end{proof}

\begin{remark}
The standard presentation~\eqref{eq:lambda-mu-explicit}, \eqref{eq:mu-matrix} of the data $(\lambda,\mu)$ allows to compute easily the number of parameters we have.
Indeed, we have a precise form of $\lambda$, and 9 parameters in $\mu$ (6 parameters $M_i$ and 3 parameters $K_i$).
On the other hand, we can rescale the form $\mu$, and also rescale and permute consistently the bases of $A_1$ and $A_2$ (i.e.\ act by the normalizer of a torus in~$\SL(A_1)$). 
This allows to kill three parameters. 
Moreover, as we will see soon (Lemma~\ref{lemma:lambda-mu-shift}), the variety $\cX_{\lambda,\mu}$ does not change if we replace the pair $(\lambda,\mu)$ by the pair $(\lambda - t\mu^2,\mu)$ for any $t \in \k$.
This allows to kill one more parameter (in fact, by using this we can assume $K_1 + K_2 + K_3 = 0$). 
So, altogether, the moduli space of varieties $\cX_{\lambda,\mu}$ is 5-dimensional.

This feature makes the situation with K\"uchle fivefolds of type $(c5)$ rather different from other examples of Fano fivefolds ($\P^5$, $\Gr(2,5) \cap H$, $(\P^1)^5$, $\Bl_{v_2(\P^2)}\P^5$), 
whose half-anticanonical section has a noncommutative K3 surface 
\end{remark}

%
%



\section{A description of K\"uchle fivefolds}\label{section:5folds}

\subsection{The main Theorem}\label{subsection:main-theorem}

Recall the definition of the K\"uchle fivefold of type $(c5)$ from the introduction. We start with a 7-dimensional vector space $W$, a 2-form $\mu \in \bw2W^\vee$ and a 4-form $\lambda \in \bw4W^\vee$. 
We consider the Grassmannian $\Gr(3,W)$ of 3-dimensional subspaces in $W$ with its tautological subbundles $\cU_3 \subset W \otimes \cO$ and $\cU_3^\perp \subset W^\vee \otimes \cO$. 
By Bott Theorem $\mu$ can be considered as a global section of the vector bundle $\cU_3(1) \cong \bw2\cU_3^\vee$ and $\lambda$ can be considered as a global section of the vector bundle $\cU_3^\perp(1) \cong \bw3(W/\cU_3)$. 
Then $\cX_{\lambda,\mu} \subset \Gr(2,W)$ is defined as the zero locus of $(\mu,\lambda)$. 
In other words, it is the subvariety of $\Gr(3,W)$ parameterizing all 3-subspaces $U_3 \subset W$ annihilated by $\lambda$ and isotropic for $\mu$.

\begin{lemma}\label{lemma:lambda-mu-shift}
For a general choice of $\lambda$ and $\mu$ the variety $\cX_{\lambda,\mu}$ is a smooth Fano fivefold of index~$2$.
Moreover, for any $t \in \k$ we have $\cX_{\lambda - t\mu^2,\mu} = \cX_{\lambda,\mu}$.
\end{lemma}
\begin{proof}
The smoothness of general $\cX_{\lambda,\mu}$ follows from Bertini Theorem since the vector bundles $\cU_3(1)$ and $\cU_3^\perp(1)$ are globally generated.
The dimension of $\cX_\lm$ is 
\begin{equation*}
\dim \Gr(3,W) - \rank(\cU_3(1)) - \rank(\cU_3^\perp(1)) = 3\cdot 4 - 3 - 4 = 5,
\end{equation*}
and by adjunction the canonical class of $\cX_\lm$ is given by
\begin{equation*}
\omega_{\cX_\lm} \cong 
(\omega_{\Gr(3,W)} \otimes \det(\cU_3(1)) \otimes \det(\cU_3^\perp(1)))\vert_{\cX_\lm} \cong
(\cO(-7) \otimes \cO(2) \otimes \cO(3))\vert_{\cX_\lm} \cong
\cO_{\cX_\lm}(-2).
\end{equation*}
Thus $\cX_\lm$ is a Fano fivefold of index 2.
If $u_1,u_2,u_3 \in U_3$ then 
\begin{equation*}
\mu^2 \conv (u_1\wedge u_2 \wedge u_3) = 
\mu(u_1,u_2) (\mu \conv u_3) - \mu(u_1,u_3) (\mu \conv u_2) + \mu(u_2,u_3) (\mu \conv u_1).
\end{equation*}
Therefore, if $U_3$ is $\mu$-isotropic, it is annihilated by $\mu^2$.
This means that $\cX_{\lambda,\mu} = \cX_{\lambda - t\mu^2,\mu}$.
\end{proof}

Now we can state the main result of the paper. 

\begin{theorem}\label{theorem:main}
Assume the pair $(\lambda,\mu)$ satisfies assumptions~\eqref{assumption:lambda-mu-general}, \eqref{assumption:pencil-regular}, \eqref{assumption:m-nonzero}, and~\eqref{assumption:lgr-hyperplane}.
Then the K\"uchle fivefold $\cX_{\lambda,\mu}$ is smooth, and there is a smooth projective variety $\tcX_{\lambda,\mu}$ and a diagram
\begin{equation}\label{diagram:fivefolds}
\vcenter{\xymatrix{
& E \ar[dl]_p \ar[r]^i & \tcX_{\lambda,\mu} \ar[dl]_\pi \ar[dr]^\bpi & \BE \ar[l]_\bi \ar[dr]^\bp \\
F \ar[r] & \cX_{\lambda,\mu} && \bcX_{\lambda,\mu}  & Z \ar[l]
}}
\end{equation}
where
\begin{itemize}
\item 
$\bcX_{\lambda,\mu} = \LGr_\mu(3,\BW)_{\blam}$ is a hyperplane section of the Lagrangian Grassmannian;
\item 
$F \cong \Fl(1,2;A_1) \cong \Fl(1,2;A_2)$ is the flag variety;
\item 
$Z$ is a smooth scroll over a del Pezzo surface of degree $6$;
\item 
the maps $\pi$ and $\bpi$ are blowups with centers in $F$ and $Z$;
\item $E$ and $\BE$ are the exceptional divisors of the blowups; and
\item the maps $i$ and $\bi$ are the embeddings of the exceptional divisors.
\end{itemize}
\end{theorem}

The proof of the Theorem takes the rest of the section. More details about the structure of the diagram will come in the course of the proof.

Throughout the section we adopt generality assumptions~\eqref{assumption:lambda-mu-general}, \eqref{assumption:pencil-regular}, \eqref{assumption:m-nonzero}, and~\eqref{assumption:lgr-hyperplane}
and use the structural results from section~\ref{subsection:genericity-assumptions}.

\subsection{The odd and even symplectic Grassmannians}\label{subsection:odd-sgr}

We start by considering the zero locus of the section $\mu$ of the vector bundle $\cU_3(1) \cong \bw2\cU_3^\vee$ on $\Gr(3,W)$.
We denote this variety 
\begin{equation*}
\LGr_\mu(3,W) \subset \Gr(3,W).
\end{equation*}
It parameterizes $\mu$-isotropic subspaces of $W$ and, for $\mu$ of corank 1 it is smooth and is usually called an {\sf odd symplectic Grassmannian}. 
There is a standard way to relate this variety to a usual (even) symplectic Grassmannian.

Recall that $w_0 \in W$ denotes a generator for the one-dimensional kernel space $\Ker\mu$ of the form $\mu$, and we have a direct sum decomposition $W = \k w_0 \oplus \BW$ of~\eqref{eq:bw}.
Note that the restriction of the form $\mu$ to $\BW$ is non-degenerate.
Denote by 
\begin{equation*}
\pr:W \to \BW
\end{equation*}
the projection along $w_0$. For each $\mu$-isotropic subspace $U_3 \subset W$ the subspace $\pr(U_3) \subset \BW$ is also $\mu$-isotropic, but its 
dimension may be either 3 (typically), or 2 (if $U_3$ contains $w_0$). Thus, the map $\pr$ induces a rational map
\begin{equation*}
\LGr_\mu(3,W) \dashrightarrow \LGr_\mu(3,\BW),
\end{equation*}
to the (even) symplectic Lagrangian Grassmannian. To resolve its indeterminacy it is natural to consider the following odd symplectic flag variety
\begin{equation*}
\LFl_\mu(3,4;W) \subset \Fl(3,4;W)
\end{equation*}
defined as the zero locus of the global section $\mu$ of the vector bundle $\bw2\cU_4^\vee$ on $\Fl(3,4;W)$
(i.e.\ the variety of flags $U_3 \subset U_4 \subset W$ such that $U_4$, and hence a fortiori $U_3$, is $\mu$-isotropic).

In what follows we denote by
\begin{equation*}
h = c_1(\cU_3^\vee) \in \Pic(\LGr_\mu(3,W)) 
\qquad\text{and}\qquad 
\bh = c_1(\bcU_3^\vee) \in \Pic(\LGr_\mu(3,\BW))
\end{equation*}
the ample generators of the Picard groups and their pullbacks to other varieties.

\begin{lemma}\label{lemma:sfl}
The map $\pr$ induces a regular map 
\begin{equation*}
\bpi \colon \LFl_\mu(3,4;W) \to \LGr_\mu(3,\BW)
\end{equation*}
which is a $\P^3$-fibration. The forgetful map 
\begin{equation*}
\pi \colon \LFl_\mu(3,4;W) \to \LGr_\mu(3,W)
\end{equation*}
is the blowup with center in the subvariety isomorphic to $\LGr_\mu(2,\BW)$.
\end{lemma}
\begin{proof}
Let $U_4 \subset W$ be a $\mu$-isotropic subspace. 
Then the subspace $\pr(U_4) \subset \BW$ is also $\mu$-isotropic.
But since the form $\mu$ on $\BW$ is nondegenerate, we have $\dim(\pr(U_4)) \le 3$, hence any such $U_4$ contains $w_0$, and the space $\pr(U_4) \cong U_4/w_0 \subset \BW$ is Lagrangian.
This shows that the projection $\pr$ induces an isomorphism 
\begin{equation}\label{eq:lg4w-lg3bw}
\LGr_\mu(4,W) \cong \LGr_\mu(3,\BW),
\end{equation}
under which the tautological bundles are related by the (canonically split) exact sequence
\begin{equation}\label{eq:u4-bu3}
0 \to \cO \xrightarrow{\ w_0\ } \cU_4 \xrightarrow{\quad} \bcU_3 \to 0.
\end{equation}
Consequently, the projection $\LFl_\mu(3,4;W) \to \LGr_\mu(4,W)$ can be understood as a map 
\begin{equation*}
\bpi: \LFl_\mu(3,4;W) \to \LGr_\mu(3;\BW).
\end{equation*}
%
Since $\Fl(3,4;W) \cong \Gr_{\Gr(4,W)}(3,\cU_4)$ and the vector bundle $\bw2\cU_4^\vee$ is a pullback from $\Gr(4,W)$, it follows that
\begin{equation*}
\LFl_\mu(3,4;W) \cong 
\Gr_{\LGr_\mu(3,\BW)}(3,\cO \oplus \bcU_3) \cong
\P_{\LGr_\mu(3,\BW)}(\bw3(\cO \oplus \bcU_3)) 
\end{equation*}
is a $\P^3$-bundle. So, it remains to describe the forgetful map $\pi$.

Consider the zero locus of the global section $w_0$ of the vector bundle $W/\cU_3$ on $\LGr_\mu(3,W)$. 
The above arguments show that the projection $\pr$ induces an isomorphism of this zero locus with the symplectic Grassmannian $\LGr_\mu(2,\BW)$ under which the tautological bundles are related by the (canonically split) exact sequence
\begin{equation}\label{eq:u3-bu2}
0 \to \cO \xrightarrow{\ w_0\ } w_0^*\cU_3 \xrightarrow{\quad} \bcU_2 \to 0,
\end{equation}
where we denote the corresponding embedding by
\begin{equation*}\label{eq:map-w0}
w_0 \colon \LGr_\mu(2,\BW) \hookrightarrow \LGr_\mu(3,W).
\end{equation*}
Let us show that the map $\pi:\LFl_\mu(3,4;W) \to \LGr_\mu(3,W)$ is the blowup of the subvariety $w_0(\LGr_\mu(2,\BW)) \subset \LGr_\mu(3,W)$.
For this we use the blowup Lemma~\ref{lemma:blowup}.

Note that $\LFl_\mu(3,4;W)$ embeds into $\P_{\LGr_\mu(3,W)}(W/\cU_3)$ as the zero locus of a section of the vector bundle $\cU_3^\vee \otimes \cO(1)$ 
corresponding to the morphism of vector bundles $\mu: W/\cU_3 \to \cU_3^\vee$ induced by the skew form $\mu$.
It is easy to see that the degeneracy locus $D(\mu)$ of this map is supported (at least set-theoretically) on $w_0(\LGr_\mu(2,\BW))$, hence has codimension 2, and that the second degeneracy locus $D_2(\mu)$ is empty.
Hence, Lemma~\ref{lemma:blowup} applies and shows that $\LFl_\mu(3,4;W)$ is the blowup of $\LGr_\mu(3,W)$ with center in $D(\mu)$,
and the exceptional divisor is isomorphic to the projectivization of the vector bundle $\Ker(W/\cU_3 \xrightarrow{\ \mu\ } \cU_3^\vee)$ on $D(\mu)$.

So, it remains to show that the degeneracy locus $D(\mu)$ of the morphism $W/\cU_3 \to \cU_3^\vee$ is equal to $w_0(\LGr_\mu(2,\BW))$ scheme-theoretically.
For this recall that, by the proof of Lemma~\ref{lemma:blowup}, the dual of this map extends to an exact sequence
\begin{equation}\label{eq:deg-loc-sgr37}
0 \to \cU_3 \to \cU_3^\perp \to \cO \to \cO_D \to 0
\end{equation}
(since $\det(W/\cU_3) \cong \det(\cU_3^\vee)$), hence $D(\mu)$ is the zero locus of a global section of the vector bundle $W/\cU_3$.
This section, evidently, corresponds to the vector $w_0$, and its zero locus was already shown to be equal to $w_0(\LGr_\mu(2,\BW))$.
\end{proof}

\subsection{The zero locus of $\lambda$}\label{subsection:zero-locus-lambda}

Recall that $\cX_{\lambda,\mu}$ is the zero locus of $\lambda$ considered as a global section of the bundle $\cU_3^\perp(h)$ on $\LGr_\mu(3,W)$.
We define 
\begin{equation*}
\tcX_{\lambda,\mu} \subset \LFl_\mu(3,4;W)
\end{equation*}
as the zero locus of $\lambda$, considered as the global section of the pullback of this vector bundle to the odd symplectic flag variety.
Consider the composition of the maps
\begin{equation*}
\tcX_\lm \hookrightarrow \LFl_\mu(3,4;W) \cong \P_{\LGr_\mu(3,\BW)}(\bw3\cU_4) \twoheadrightarrow \LGr_\mu(3,\BW),
\end{equation*}
where the isomorphism in the middle follows from~\eqref{eq:lg4w-lg3bw}.

\begin{lemma}\label{lemma:tcx-zero-loc}
The map $\tcX_\lm \to \LGr_\mu(3,\BW)$ factors as the composition
\begin{equation*}
\tcX_\lm \xrightarrow{\ \bpi\ } \LGr_\mu(3,\BW)_\blam \hookrightarrow \LGr_\mu(3,\BW),
\end{equation*}
where $\LGr_\mu(3,\BW)_\blam$ is the hyperplane section of $\LGr_\mu(3,\BW)$ given by the $3$-form $\blam \in \bw3\BW^\vee$.
Moreover, $\tcX_\lm \subset \P_{\LGr_\mu(3,\BW)_\blam}(\bw3\cU_4)$ is the zero locus of a global section of the bundle $\cU_4^\perp(h)$.
\end{lemma}
\begin{proof}
Dualizing the tautological sequence $0 \to \cU_4/\cU_3 \to W/\cU_3 \to W/\cU_4 \to 0$, twisting it by~$\cO(h)$, and taking into account an isomorphism 
\begin{equation}\label{eq:u4-u3}
\cU_4/\cU_3 \cong \det(\cU_4) \otimes \det(\cU_3^\vee) \cong \det(\bcU_3) \otimes \det(\cU_3^\vee) \cong \cO(h - \bh)
\end{equation}
(with~\eqref{eq:u4-bu3} used in the second isomorphism), we get an exact sequence
\begin{equation}\label{eq:u4p-u3p}
0 \to \cU_4^\perp(h) \to \cU_3^\perp(h) \to \cO(\bh) \to 0.
\end{equation}
Since $\tcX_\lm \subset \LFl_\mu(3,4;W) \cong \P_{\LGr_\mu(3,\BW)}(\bw3\cU_4)$ is the zero locus of a section of $\cU_3^\perp(h)$, it lies in the zero locus of the induced section of $\cO(\bh)$.
The line bundle $\cO(\bh)$ is a pullback from $\LGr_\mu(3,\BW)$, hence a zero locus of its section is the preimage of the zero locus of a hyperplane section of $\LGr_\mu(3,\BW)$.
%
By definition, evaluation of this section on the subspace $\BU_3 \subset \BW$ is equal to the evaluation of the 4-form $\lambda$ on the corresponding subspace $U_4 = \k w_0 \oplus \BU_3$ of $W$.
Since the convolution of $\lambda$ with $w_0$ is $\blam$, this is equal to the evaluation of $\blam$ on $\BU_3$. 
Thus the hyperplane section of $\LGr_\mu(3,\BW)$ we are interested in is given by $\blam$.

It remains to note that when restricted to the zero locus $\P_{\LGr_\mu(3,\BW)_\blam}(\bw3\cU_4)$ of the section of $\cO(\bh)$, the section $\lambda$ of $\cU_3^\perp(h)$ comes (via~\eqref{eq:u4p-u3p}) from a section of $\cU_4^\perp(h)$, 
and $\tcX_\lm$ is the zero locus of this section.
\end{proof}

Consider the following diagram of vector bundles on $\LGr_\mu(3,\BW)_\blam$.
\begin{equation}\label{diagram:hlam}
\vcenter{\xymatrix{
& \bw3\cU_4 \ar[r] \ar@{-->}[d]_\hlam & \bw3W \otimes \cO \ar[d]^\lambda \\
0 \ar[r] & \cU_4^\perp \ar[r] & W^\vee \otimes \cO \ar[r] & \cU_4^\vee \ar[r] & 0
}}
\end{equation}
Here the bottom line is the tautological exact sequence, and since the composition of the top arrow with the map $\lambda$ and the projection $W^\vee \otimes \cO \to \cU_4^\vee$ vanishes 
(by definition of $\LGr_\mu(3;\BW)_\blam$), there is a dashed arrow $\hlam$ on the left, making the diagram commutative.
It is clear that the map 
\begin{equation}\label{eq:map-hlam}
\hlam \colon \bw3\cU_4 \to \cU_4^\perp
\end{equation}
defined in this way, corresponds to the section of the vector bundle $\cU_4^\perp(h)$ on $\P_{\LGr_\mu(3,\BW)_\blam}(\bw3\cU_4)$ cutting out $\tcX_\lm$.
So, we can use the blowup Lemma~\ref{lemma:blowup} to describe $\tcX_\lm$ as the blowup of $\LGr_\mu(3,\BW)_\blam$.
For this, however, we need to control the degeneracy loci of $\hlam$.

%
%
%
%

\subsection{Degeneracy loci of the morphism $\hlam$}\label{subsection:deg-loc-hlam}

We start by discussing the degeneracy loci of the morphism $\hlam$ on the hyperplane section $\Gr(4,W)_\lambda$ of the Grassmannian (where it is defined via the same diagram~\eqref{diagram:hlam}), and
later we will restrict to the Lagrangian Grassmannian $\LGr_\mu(3,\BW) \subset \Gr(4,W)$.

Recall the definition of the $\Gtwo$-adjoint variety $\Gr_\lambda(5,W) \subset \Gr(5,W)$ and the quadric $\bQ^\vee_\lambda \subset \P(W^\vee)$ from section~\ref{subsection:4form-4space}.

\begin{proposition}\label{proposition:gr-deg-loc}
Let $\Gr(4,W)_{\lambda,k} \subset \Gr(4,W)_\lambda$ be the $k$-th degeneracy locus of the map~\eqref{eq:map-hlam} on the Grassmannian hyperplane section $\Gr(4,W)_\lambda$.
Then $\Gr(4,W)_{\lambda,3} = \emptyset$ and

\noindent$(a)$ there is an isomorphism $\Gr(4,W)_{\lambda,2} \cong \bQ^\vee_\lambda$;

\noindent$(b)$ there is a dominant regular birational morphism $\Gr_{\Gr_\lambda(5,W)}(4,\cU_5) \to \Gr(4,W)_{\lambda,1}$, which is an isomorphism over the complement of $\Gr(4,W)_{\lambda,2}$.
\end{proposition}
\begin{proof}
We start with part $(b)$. 
Assume the map $\hlam\colon \bw3\cU_4 \to \cU_4^\perp$ is degenerate at point $U_4$ of $\Gr(4,W)$. 
Then its image is contained in a 2-dimensional subspace of $U_4^\perp$. 
Every such subspace can be written as $U_5^\perp$ for some 5-dimensional subspace $U_5 \subset W$ containing $U_4$.
The condition $\hlam(\bw3U_4) \subset U_5^\perp$ with the condition $\lambda(\bw4U_4) = 0$ (defining the hyperplane section $\Gr(4,W)_\lambda \subset \Gr(4,W)$), together imply that $U_5$ is $\lambda$-isotropic.
Thus $U_5 \in \Gr_\lambda(5,W)$. 

Conversely, for every $U_5 \in \Gr_\lambda(5,W)$, any subspace $U_4 \subset U_5$ gives a point of $\Gr(4,W)_{\lambda,1}$. 
This means that the natural map $\Gr_{\Gr_\lambda(5,W)}(4,\cU_5) \to \Gr(4,W)$ is surjective onto the degeneracy locus $\Gr(4,W)_{\lambda,1}$.
Moreover, if $U_4 \in \Gr(4,W)_{\lambda,1} \setminus \Gr(4,W)_{\lambda,2}$, then the subspace $U_5$ is determined uniquely by~$U_4$, hence the constructed map is an isomorphism over the complement of $\Gr(4,W)_{\lambda,2}$.

Now assume $U_4 \in \Gr(4,W)_{\lambda,2}$. 
Then the same argument as above shows that there is a 6-dimensional subspace $U_6 \subset W$ such that for any $U_5$ such that $U_4 \subset U_5 \subset U_6$ the subspace $U_5$ is $\lambda$-isotropic.
Let $w^\vee \in \P(W^\vee)$ be the point corresponding to the subspace $U_6$.
Since $w^\vee \in U_5^\perp$, it follows from Remark~\ref{remark:gr27-g2} that $w^\vee \in \bQ^\vee_\lambda$.


Conversely, if $U_6 \subset W$ is the subspace corresponding to a point $w^\vee \in \bQ^\vee_\lambda$ then the bivector $\lambda^\vee \conv w^\vee$ is degenerate, and moreover has rank 4.
Therefore, its kernel is 2-dimensional and can be written as $U_4^\perp$ for a unique subspace $U_4 \subset U_6$.
Then $U_5 \subset U_6$ is $\lambda$-isotropic if and only if $U_4 \subset U_5$.
Hence the image of $\hlam$ at $U_4$ is $U_6^\perp$.
Altogether, this means that there is a surjective map $\bQ^\vee_\lambda \to \Gr(4,W)_{\lambda,2}$, which is an isomorphism over the complement of $\Gr(4,W)_{\lambda,3}$.
So, it remains to show that $\Gr(4,W)_{\lambda,3} = \emptyset$.

Assume that $U_4 \in \Gr(4,W)_{\lambda,3}$, i.e.\ $\hlam = 0$ at $U_4$. 
Then every $U_5$ containing $U_4$ is $\lambda$-isotropic.
Therefore, the projective plane $\P(W/U_4) \subset \Gr(5,W)$ is contained in $\Gr_\lambda(5,W)$.
But the $\Gtwo$-adjoint variety $\Gr_\lambda(5,W)$ does not contain planes by \cite{kapustka2013genus10}.
This contradiction shows that $\Gr(4,W)_{\lambda,3} = \emptyset$ and thus completes the proof of the Proposition.
\end{proof}

Now we go back to the symplectic situation. Denote by $\LGr_\mu(3,\BW)_{\blam,k}$ the $k$-th degeneracy locus of the morphism~\eqref{eq:map-hlam}.
Clearly,
\begin{equation*}
\LGr_\mu(3,\BW)_{\blam,k} = \LGr_\mu(3,\BW) \cap \Gr(4,W)_{\lambda,k}.
\end{equation*}
Our next goal is to describe these intersections. First, consider the case $k = 1$. 
Instead of the degeneracy locus $\Gr(4,W)_{\lambda,k}$ consider its blowup $\Gr_{\Gr_\lambda(5,W)}(4,\cU_5)$ and instead of the intersection consider the fiber product
\begin{equation}\label{eq:z-def}
Z := \LGr_\mu(3,\BW) \times_{\Gr(4,W)} \Gr_{\Gr_\lambda(5,W)}(4,\cU_5).
\end{equation}
Note that we have a canonical embedding
\begin{equation*}
Z \hookrightarrow \Gr_{\Gr_\lambda(5,W)}(4,\cU_5) \hookrightarrow \Fl(4,5;W).
\end{equation*}
Thus $Z$ parameterizes some flags $(U_4,U_5)$ of subspaces of $W$ such that $U_5$ is $\lambda$-isotropic and $U_4$ is $\mu$-isotropic.
As usual, we denote by $\cU_4 \hookrightarrow \cU_5 \hookrightarrow W \otimes \cO_Z$ the pullback to $Z$ of the tautological flag.

\begin{proposition}\label{proposition:s-z}
The fiber product $Z$ is isomorphic to a $\P^1$-bundle over a sextic del Pezzo surface $S \subset \P(\bw2A_1) \times \P(\bw2A_2)$.
\end{proposition}
\begin{proof}
By definition $Z$ parameterizes pairs $(U_4,U_5)$, where $U_5 \subset W$ is a $\lambda$-isotropic subspace, and $U_4 \subset U_5$ is a $\mu$-isotropic subspace.
In particular, $U_4$ and hence also $U_5$ contains $w_0$.
So, one can rephrase the definition of $Z$ by saying that it parameterizes all pairs $(\BU_3,\BU_4)$, where $\BU_3 \subset \BU_4 \subset \BW$,
the subspace $\BU_4$ is both $\blam$ and $\lambda'$-isotropic, and $\BU_3$ is $\mu$-isotropic.
Since $\dim\BU_3 = 3$ and $\dim\BU_4 = 4$, the last condition means 
that the restriction of $\mu$ to $\BU_4$ is degenerate, hence the 4-form $\mu^2 = \mu \wedge \mu$ vanishes on $\BU_4$, i.e.\ $\BU_4$ is also $\mu^2$-isotropic. 
So, consider the subvariety 
\begin{equation*}
S \subset \Gr(4,\BW),
\end{equation*}
parameterizing subspaces $\BU_4 \subset \BW$ that are isotropic with respect to $\blam$, $\lambda'$ and $\mu^2$.
The above discussion shows that there is a map
\begin{equation*}
\sigma: Z \to S,
\qquad 
(\BU_3,\BU_4) \mapsto \BU_4.
\end{equation*}
We will show first that $S$ is a sextic del Pezzo surface, and then that $\sigma$ is a $\P^1$-bundle.

First, the locus of $\blam$-isotropic subspaces $\BU_4 \subset \BW$ is the zero locus of the section $\blam$ of the vector bundle $\bw3\bcU_4^\vee$ on $\Gr(4,\BW)$.
By Corollary~\ref{corollary:gr46-3form} it is isomorphic to 
\begin{equation*}
\Gr(2,A_1) \times \Gr(2,A_2) \cong \P(\bw2A_1) \times \P(\bw2A_2) \cong \P^2 \times \P^2.
\end{equation*}
By~\eqref{eq:gr46-3form-bu4} we have $\bw4\bcU_4^\vee \cong \bw2\cU_{2,A_1}^\vee \otimes \bw2\cU_{2,A_2}^\vee \cong \cO_{\P(\wedge^2A_1) \times \P(\wedge^2A_2)}(1,1)$.
Therefore, the zero loci of $\lambda'$ and $\mu^2$ on it are two divisors of bidegree $(1,1)$ in the above product corresponding to the pairings between $\bw2A_1$ and $\bw2A_2$ given by $\lambda'$ and $\mu^2$,
and $S$ is the intersection of these two divisors.
It is well known that an intersection of two such divisors is a smooth surface if and only if the line spanned by their equations in the space $\P(\bw2A_1^\vee \otimes \bw2A_2^\vee)$ is transversal to the divisor of degenerate pairings.
By assumption~\eqref{assumption:pencil-regular} transversality holds for the pencil generated by $\lambda'$ and $\mu^2$, hence $S$ is a smooth surface.
By adjunction its canonical class is the restriction of $\cO(-1,-1)$, and its anticanonical degree equals the degree of $\P^2 \times \P^2$, which is equal to~6. 
So, finally we see that $S$ is a sextic del Pezzo surface.

\begin{remark}\label{remark:s}
Note that the projections $S \to \P(\bw2A_1)$ and $S \to \P(\bw2A_2)$ are birational and provide two representations of $S$ as a blowup of $\P^2$ in three points. 
If we choose a basis of $\BW$ in which $\lambda'$ and $\mu^2$ have form~\eqref{eq:lambda-mu-explicit}, then the centers of the blowups are given
by the points $\{e_{12},e_{13},e_{23}\} \in \P(\bw2A_1)$ and $\{e_{45}, e_{46}, e_{56}\} \in \P(\bw2A_2)$ respectively.

We denote by $\bh_1$ and $\bh_2$ the pullbacks to $S$ of the hyperplane classes of $\P(\bw2A_1)$ and $\P(\bw2A_2)$ respectively.
Then
\begin{equation}\label{eq:s-intersection}
\bh_1^2 = \bh_2^2 = 1,
\qquad 
\bh_1\cdot\bh_2 = 2
\end{equation}
and $\bh_1 + \bh_2$ is the anticanonical class on $S$.
\end{remark}

Now consider the fiber of the map $\sigma:Z \to S$ over a point of $S$ corresponding to a subspace $\BU_4 \subset \BW$.
By definition of $Z$, it parameterizes all $\mu$-isotropic 3-subspaces in $\BU_4$.
The kernel of the restriction of $\mu$ to $\BU_4$ is equal to the intersection $\BU_4 \cap \BU_4^{\perp\mu}$, where $\BU_4^{\perp\mu} \subset \BW$ is the orthogonal of $\BU_4$ with respect to $\mu$.
Since $\mu$ is nondegenerate on $\BW$, the orthogonal is 2-dimensional.
On the other hand, the rank of $\mu$ on $\BU_4$ is even, hence the dimension of $\BU_4 \cap \BU_4^{\perp\mu}$ is even.
Thus $\mu$ is degenerate on $\BU_4$ if and only if $\BU_4^{\perp\mu} \subset \BU_4$ (i.e.\ $\BU_4$ is coisotropic), and then $\BU_4^{\perp\mu}$ is the kernel of the restriction of $\mu$ to $\BU_4$.
A 3-subspace in such $\BU_4$ is isotropic if and only if it contains the kernel $\BU_4^{\perp\mu}$ of the restriction of $\mu$, hence the fiber of $\sigma$ over $\BU_4$ is the projective line $\P(\BU_4/\BU_4^{\perp\mu}) \cong \P^1$.

The argument above can be rephrased by saying that $Z$ is the projectivization of the following vector bundle over $S$.
Let $\cU_{2,A_1}$ and $\cU_{2,A_2}$ be the pullbacks to $S$ of the tautological vector bundles on $\Gr(2,A_1)$ and $\Gr(2,A_2)$. 
Note that the corresponding quotient bundles are $\cO(\bh_1)$ and $\cO(\bh_2)$ respectively:
\begin{equation*}
0 \to \cU_{2,A_1} \to A_1 \otimes \cO \to \cO(\bh_1) \to 0,
\qquad
0 \to \cU_{2,A_2} \to A_2 \otimes \cO \to \cO(\bh_2) \to 0.
\end{equation*}
By Corollary~\ref{corollary:gr46-3form} and~\eqref{eq:gr46-3form-bu4} we have $\bcU_4 = \cU_{2,A_1} \oplus \cU_{2,A_2}$.
Summing up the above sequences we deduce that $\BW/\bcU_4 \cong \cO(\bh_1) \oplus \cO(\bh_2)$,
hence $\bcU_4^{\perp\mu} \cong \cO(-\bh_1) \oplus \cO(-\bh_2)$, and so
\begin{equation}\label{eq:z-proj}
Z \cong \P_S(\cV_S),
\end{equation}
where the vector bundle $\cV_S$ is defined by an exact sequence
\begin{equation}\label{eq:cvs}
0 \to \cO(-\bh_1) \oplus \cO(-\bh_2) \to \cU_{2,A_1} \oplus \cU_{2,A_2} \to \cV_S \to 0
\end{equation} 
on $S$. 
Note also that the bundle $\cV_S$ is self-dual.
\end{proof}


Let us fix some details of the description of $Z$ established in the proof of Proposition~\ref{proposition:s-z}.
First, the natural embedding $Z \to \Fl(4,5;W)$ factors as the composition
\begin{equation*}
Z \hookrightarrow \Fl(3,4;\BW) \hookrightarrow \Fl(4,5;W),
\end{equation*}
where the second map takes a flag $(\BU_3,\BU_4)$ in $\BW$ to the flag $(\BU_3 \oplus \k w_0, \BU_4 \oplus \k w_0)$ in $W$.
In particular, if $\bcU_3 \hookrightarrow \bcU_4 \hookrightarrow \BW \otimes \cO$ is the restriction to $Z$ of the tautological flag on $\Fl(3,4;\BW)$, then
the restriction to $Z$ of the tautological flag from $\Fl(4,5;W)$ is given by
\begin{equation}\label{eq:z-u4-u5}
\cU_4 \cong \bcU_3 \oplus \cO
\qquad\text{and}\qquad
\cU_5 \cong \bcU_4 \oplus \cO.
\end{equation}
Moreover, 
there is a commutative diagram
\begin{equation*}
\vcenter{\xymatrix{
Z \ar@{^{(}->}[rr] \ar[d]_\sigma && \Fl(3,4;\BW)  \ar[d] \\
S \ar@{^{(}->}[r] & \Gr(2,A_1) \times \Gr(2,A_2) \ar@{^{(}->}[r] & \Gr(4,\BW)
}}
\end{equation*}
and if $\cU_{2,A_1}$ and $\cU_{2,A_2}$ denote the tautological bundles on $\Gr(2,A_1)$ and $\Gr(2,A_2)$ then
\begin{equation}\label{eq:z-bu4}
\bcU_4 \cong \cU_{2,A_1} \oplus \cU_{2,A_2}.
\end{equation} 
Denote by $\sv_Z$ the hyperplane class for the projectivization $Z \cong \P_S(\cV_S)$, see~\eqref{eq:z-proj}. Then we have an exact sequence
\begin{equation}\label{eq:bcu3-1}
0 \to \cO(-\bh_1) \oplus \cO(-\bh_2) \to \bcU_3 \to \cO(-\sv_Z) \to 0
\end{equation} 
and since $\cV_S$ is self-dual, also
\begin{equation}\label{eq:bcu3-2}
0 \to \bcU_3 \to \bcU_4 \to \cO(\sv_Z) \to 0.
\end{equation}

\subsection{A description of the map $\hlam$ on $Z$}\label{subsection:hlam-z}

In this section we discuss the pullback to $Z$ of the map $\hlam$ defined by~\eqref{eq:map-hlam}.
In particular, we show that on $Z$ it has a constant rank equal to~2.
As we will see, this gives a description of the degeneracy loci of $\hlam$ on $\LGr_\mu(3,\BW)_\blam$.

\begin{lemma}\label{lemma:deg-loc-u2-vs}
On the surface $S \subset \Gr(2,A_1) \times \Gr(2,A_2)$ consider the compositions of the maps 
\begin{equation}\label{eq:maps-l1-l2}
\cU_{2,A_1} \hookrightarrow \cU_{2,A_1} \oplus \cU_{2,A_2} \twoheadrightarrow \cV_S
\qquad\text{and}\qquad
\cU_{2,A_2} \hookrightarrow \cU_{2,A_1} \oplus \cU_{2,A_2} \twoheadrightarrow \cV_S
\end{equation}
and let $\BC_1 \subset S$ and $\BC_2 \subset S$ be their degeneracy loci.
Then $\BC_1$ is a smooth rational curve in the linear system $|\bh_1|$, $\BC_2$ is a smooth rational curve in the linear system $|\bh_2|$,
and the ranks of the maps~\eqref{eq:maps-l1-l2} on $\BC_1$ and $\BC_2$ are equal to $1$.
\end{lemma}
\begin{proof}
Consider the first map (the second map is dealt with analogously).
Since the kernel of the map $\cU_{2,A_1} \oplus \cU_{2,A_2} \to \cV_S$ by definition of $\cV_S$ is the kernel of the restriction of the skew-form $\mu$ to $\bcU_4$,
the subscheme $\BC_1 \subset S$, parameterizes pairs $(U_{2,A_1},U_{2,A_2})$ such that $U_{2,A_1}$ intersects the kernel of $\mu$.
If this holds, then of course $U_{2,A_1}$ is $\mu$-isotropic. 
Conversely, if $U_{2,A_1}$ is $\mu$-isotropic, it has to intersect the kernel space of $\mu$ on $\bcU_4$ (since otherwise $\mu$ would be identically zero on~$\bcU_4$, and this contradicts to the no-degeneracy of $\mu$ on $\BW$).
Thus $\BC_1$ is the locus of points of $S$, such that $U_{2,A_1}$ is $\mu$-isotropic.
This condition, in its turn, is equivalent to the condition $\mu_1 \in U_{2,A_1}$, where $\mu_1 \in A_1$ is the kernel vector of the restriction of $\mu$ to $A_1$.
Therefore, $\BC_1$ is the preimage of the line in $\P(\bw2A_1)$ corresponding to the vector $\mu_1 \in A_1 \cong \bw2A_1^\vee$ under the projection $S \to \P(\bw2A_1)$.
Hence it is a curve in the linear system~$|\bh_1|$. 

Note that the line corresponding to $\mu_1$ does not pass through the centers of the blowup $S \to \P(\bw2A_1)$, hence $\BC_1$ is a smooth curve.
Indeed, by Remark~\ref{remark:s} in a standard basis the center of this blowup is the triple of points $\{e_{12},e_{13},e_{23}\} \in \P(\bw2A_1)$, while the line is given by the equation
$\mu_1 = M_6K_1e_1 + M_5K_2e_2 + M_4K_3e_3 = M_6K_1x_{23} - M_5K_2x_{13} + M_4K_3x_{12}$.
By~\eqref{assumption:m-nonzero} all coefficients are nonzero.

Finally, note that the subspace $U_{2,A_1}$ cannot be equal to the kernel of $\mu$ on~$\bcU_4$, since otherwise the subspaces $U_{2,A_1} \subset A_1$ and $U_{2,A_2} \subset A_2$
would be $\mu$-orthogonal, which would contradict to the nondegeneracy of the pairing between $A_1$ and $A_2$ induced by $\mu$.
Therefore, the rank of the map $\cU_{2,A_1} \to \cV_S$ is always nonzero, so on $\BC_1$ it is identically equal to~1.
\end{proof}


Recall that $\sv_Z$ denotes the relative hyperplane class of $Z = \P_S(\cV_S)$.

\begin{corollary}\label{corollary:zero-u2-vz}
On $Z$ consider the compositions 
\begin{equation*}
\cU_{2,A_1} \to \sigma^*\cV_S \twoheadrightarrow \cO_Z(\sv_Z)
\qquad\text{and}\qquad
\cU_{2,A_2} \to \sigma^*\cV_S \twoheadrightarrow \cO_Z(\sv_Z)
\end{equation*}
of the maps~\eqref{eq:maps-l1-l2} with the canonical epimorphisms, and let $C_1 \subset Z$ and $C_2 \subset Z_2$ be their zero loci.
Then the map $\sigma:Z \to S$ induces isomorphisms 
\begin{equation*}
C_1 \cong \BC_1
\qquad\text{and}\qquad
C_2 \cong \BC_2,
\end{equation*}
so $C_1$ and $C_2$ are smooth rational curves. Moreover, $C_1 \cap C_2 = \emptyset$.
\end{corollary}
\begin{proof}
The first statement follows from the fact that the maps~\eqref{eq:maps-l1-l2} both have rank 1 on $\BC_1$ and~$\BC_2$ respectively.
For the second, note that a point of intersection $C_1 \cap C_2$ should correspond to a subspace $\BU_3 \subset \BU_4$ such that $U_{2,A_1} \subset \BU_3$ and $U_{2,A_2} \subset \BU_3$. 
But this is of course impossible since $\dim\BU_3 = 3$.
\end{proof}

Consider the pullback of the map $\hlam$ to $Z$.
Using~\eqref{eq:z-u4-u5} we see that
\begin{equation*}
\bw3\cU_4 \cong \bw3(\bcU_3 \oplus \cO) \cong \bw3\bcU_3 \oplus \bw2\bcU_3
\end{equation*}
(where the second summand embeds into $\bw3\cU_4$ by the wedge product with $w_0$), and
\begin{equation*}
\cU_5^\perp \cong \bcU_4^\perp \cong \cU_{2,A_1}^\perp \oplus \cU_{2,A_2}^\perp \cong \cO(-\bh_1) \oplus \cO(-\bh_2).
\end{equation*}
Therefore, the map $\hlam$ can be rewritten as
\begin{equation}\label{eq:hlam-z-rewritten}
\hlam\vert_Z = \lambda' \oplus \blam : \bw3\bcU_3 \oplus \bw2\bcU_3 \xrightarrow{\ \qquad\ } \cO(-\bh_1) \oplus \cO(-\bh_2)
\end{equation}
where $\lambda'$ is considered as a map from the first summand in the left hand side, and $\blam$ as a map from the second summand.
We will prove that the map $\hlam$ is surjective (hence its rank is equal to~2) everywhere on $Z$, and this will give a complete description of the degeneracy loci of $\hlam$ on $\LGr_\mu(3,\BW)_\blam$.
For this we first analyze the rank and the cokernel of the component~$\blam$, and then we show that the component $\lambda'$ maps surjectively onto the cokernel of $\blam$.

\begin{lemma}\label{lemma:deg-loc-blam}
The map $\blam: \bw2\bcU_3 \to \cO(-\bh_1) \oplus \cO(-\bh_2)$ on $Z$ is surjective away of the complement of the curve $C_1 \sqcup C_2$.
Moreover, it extends to an exact sequence
\begin{equation}\label{eq:exact-seq-blam}
0 \to \cO(-2\sv_Z - \bh_1 - \bh_2) \to \bw2\bcU_3 \xrightarrow{\ \blam\ } \cO(-\bh_1) \oplus \cO(-\bh_2) \to \cO_{C_2}(-2) \oplus \cO_{C_1}(-2) \to 0,
\end{equation} 
and the right arrow is the direct sum of two restriction maps $\cO(-\bh_1) \to \cO(-\bh_1)\vert_{C_2} \cong \cO_{C_2}(-2)$ and $\cO(-\bh_2) \to \cO(-\bh_2)\vert_{C_1} \cong \cO_{C_1}(-2)$.
\end{lemma}
\begin{proof}
Recall that the 3-form $\blam$ is the sum~\eqref{eq:blam-general} of two summands $\blam_1 \in \bw3A_1^\perp$ and $\blam_2 \in \bw3A_2^\perp$.
It is easy to see that the component $\blam_1$ of the map $\blam : \bw2\bcU_3 \to \cO(-\bh_1) \oplus \cO(-\bh_2)$ can be written as the composition of the map
\begin{equation*}
\bw2\bcU_3 \hookrightarrow \bw2\bcU_4 \cong \bw2(\cU_{2,A_1} \oplus \cU_{2,A_2}) \twoheadrightarrow \bw2\cU_{2,A_1} \cong \cO(-\bh_1)
\end{equation*}
with the embedding $\cO(-\bh_1) \hookrightarrow \cO(-\bh_1) \oplus \cO(-\bh_2)$.
Therefore, the zero locus of $\blam_1$ is the locus of points, where the fiber of $\cU_{2,A_2}$ is contained in the fiber of $\bcU_3$,
hence is equal to the zero locus of the composition 
\begin{equation*}
\cU_{2,A_2} \hookrightarrow \bcU_4 \twoheadrightarrow \bcU_4/\bcU_3.
\end{equation*}
By~\eqref{eq:bcu3-2} the quotient $\bcU_4/\bcU_3$ is isomorphic to $\cO(\sv_Z)$ and it is easy to see that the composition of the maps above coincides with the map of Corollary~\ref{corollary:zero-u2-vz}. 
Hence its zero locus (and thus also the zero locus of the map $\blam_1:\bw2\bcU_3 \to \cO(-\bh_1)$ is equal to the curve $C_2$. 
Analogously, the zero locus of the map $\blam_2:\bw2\bcU_3 \to \cO(-\bh_2)$ is the curve $C_1$. 
Since the curves do not intersect, the rank of the map $\blam : \bw2\bcU_3 \to \cO(-\bh_1) \oplus \cO(-\bh_2)$ on $C_1 \sqcup C_2$ is 1.

Further, let us show that away of $C_1 \sqcup C_2$ the map has rank 2. 
It follows from the above discussion that the kernel of its first component at point  $(U_{2,A_1},U_{2,A_2},\BU_3)$ is equal to $(\BU_3 \cap U_{2,A_2}) \wedge \BU_3 \subset \bw3\BU_3$, 
and the kernel of the second component is $(\BU_3 \cap U_{2,A_1}) \wedge \BU_3 \subset \bw3\BU_3$.
Since both intersections $\BU_3 \cap U_{2,A_i}$ are 1-dimensional away of $C_1 \sqcup C_2$, it follows that the kernels are 2-dimensional and distinct, 
hence their intersection is 1-dimensional (in fact, it is equal to $(\BU_3 \cap U_{2,A_1}) \wedge (\BU_3 \cap U_{2,A_2}) \subset \bw3\BU_3$), hence the rank of the map $\blam$ is 2.

So far, we have proved the first statement. For the second, we note that the kernel of the map $\blam$ is a reflexive sheaf of rank 1, hence a line bundle.
Moreover, since the map $\blam$ is surjective away of a codimension 2 locus, the first Chern class of the kernel is
\begin{equation*}
c_1(\bw2\bcU_3) - c_1(\cO(-\bh_1) \oplus \cO(-\bh_2)) = (-\bh_1-\bh_2) + (-\bh_1-\sv_Z) + (-\bh_2-\sv_Z) + (\bh_1 + \bh_2)
\end{equation*}
(we use~\eqref{eq:bcu3-1} to compute the first Chern class of $\bw2\bcU_3$),
hence the kernel has the form as in~\eqref{eq:exact-seq-blam}. 
Finally, as we have seen earlier, the cokernel of the map $\blam$ equals the direct sum of the cokernels of $\blam_1$ and $\blam_2$, hence is equal to $\cO(-\bh_1)\vert_{C_2} \oplus \cO(-\bh_2)\vert_{C_1}$.
But since the curve $C_2$ projects by $\sigma$ to the curve $\BC_2$ in the linear system $|\bh_2|$, and the intersection product $\bh_1\cdot\bh_2$ on $S$ is equal to 2, the first summand is $\cO_{C_2}(-2)$.
Analogously, the second summand is $\cO_{C_1}(-2)$.
%
%
%
%
%
%
%
%
\end{proof}

\begin{lemma}
On $Z$ the composition of the maps 
\begin{equation*}
\bw3\bcU_3 \xrightarrow{\ \lambda'\ } \cO(-\bh_1) \oplus \cO(-\bh_2) \twoheadrightarrow \Coker(\blam) \cong \cO_{C_2}(-2) \oplus \cO_{C_1}(-2)
\end{equation*}
is surjective. In particular, the map $\hlam$ is surjective on all $Z$.
\end{lemma}
\begin{proof}
The summand $\cO_{C_2}(-2)$ in the right hand side is a quotient of $\cO(-\bh_1)$. 
So, for the surjectivity over it, it is enough to show that the map 
\begin{equation*}
\lambda':\bw3\BU_3 \to U_{2,A_1}^\perp \oplus U_{2,A_2}^\perp 
\end{equation*}
has a nonzero first component.
Recall that on $C_2$ the space $\BU_3$ contains $U_{2,A_2}$, and since the form $\lambda'$ annihilates $A_2$, it follows that the image of $\lambda'$ is contained in $U_{2,A_1}^\perp$, so it is enough to show that $\lambda'$ is nonzero on $C_2$.
But this follows immediately from Lemma~\ref{lemma:no21cases}, since vanishing of $\lambda'$ on $\BU_3$ contradicts to the $\mu$-isotropicity.
The surjectivity of $\lambda'$ on $C_1$ is proved analogously, and the surjectivity of $\hlam$ follows.
\end{proof}

Recall that $Z$ by~\eqref{eq:z-def} comes with a projection to $\LGr_\mu(3,\BW) \cap \Gr(4,W)_{\lambda} = \LGr_\mu(3,\BW)_\blam$.

\begin{corollary}\label{corollary:z-deg-loc}
The map $Z \to \LGr_\mu(3,\BW)_\blam$ is a closed embedding and gives an isomorphism 
\begin{equation}\label{eq:z-deg-locus}
Z \cong \LGr_\mu(3,\BW)_{\lambda,1}.
\end{equation} 
Moreover, $\LGr_\mu(3,\BW)_{\lambda,2} = \emptyset$.
\end{corollary}
\begin{proof}
First, by definition of $Z$ and Proposition~\ref{proposition:gr-deg-loc} the map $Z \to \LGr_\mu(3,\BW)_{\lambda,1}$ is surjective.
Further, since the rank of $\hlam$ on $Z$ is identically equal to 2, so the corank is 1, it follows that $\LGr_\mu(3,\BW)_{\lambda,2}$ is not in the image, hence is empty.
Finally, by Proposition~\ref{proposition:gr-deg-loc} the map is an isomorphism over the complement of $\LGr_\mu(3,\BW)_{\lambda,2}$, hence is an isomorphism.
\end{proof}

\subsection{Proof of the main Theorem}\label{subsection:proof-main-theorem}

Recall that the subvariety $\tcX_\lm \subset \LFl_\mu(3,4;W)$ was defined as the zero locus of the section $\lambda$ of the vector bundle $\cU_3^\perp(h)$.
In particular, the projections $\LFl_\mu(3,4;W) \to \LGr_\mu(3,W)$ and $\LFl_\mu(3,4;W) \to \LGr_\mu(4,W) \cong \LGr_\mu(3,\BW)$ (see~\eqref{eq:lg4w-lg3bw})
induce two maps
$\pi:\tcX_\lm \to \cX_\lm$ and $\bpi:\tcX_\lm \to \LGr_\mu(4,\BW)_\blam$. 

\begin{proposition}\label{proposition:blowups}
The map $\bpi:\tcX_\lm \to \LGr_\mu(3,\BW)_\blam$ is the blowup with center in $Z$.
Furthermore, $\cX_\lm$ is smooth and the map $\pi:\tcX_\lm \to \cX_\lm$ is the blowup with center in a subvariety $F \subset \cX_\lm$ isomorphic to the flag variety $\Fl(1,2;3)$.
\end{proposition}
\begin{proof}
For the first statement we apply the blowup Lemma~\ref{lemma:blowup}.
By Lemma~\ref{lemma:tcx-zero-loc} we know that $\tcX_\lm$ is the zero locus of a global section of the vector bundle $\cU_4^\perp(h)$ on $\P_{\LGr_\mu(3;\BW)_\blam}(\bw3\cU_4)$
that corresponds to the morphism~\eqref{eq:map-hlam} from a rank 4 vector bundle to a rank 3 vector bundle. 
Moreover, by Corollary~\ref{corollary:z-deg-loc} its degeneracy locus is equal to the subscheme $Z$ and thus has codimension 2 (and the higher degeneracy loci are empty).
Hence Lemma~\ref{lemma:blowup} applies and proves that $\tcX_\lm$ is the blowup of $\LGr_\mu(3;\BW)_\blam$ with center in $Z$.

Since $\LGr_\mu(3,\BW)_\blam$ is smooth by~\eqref{assumption:lgr-hyperplane} and $Z$ is smooth by Proposition~\ref{proposition:s-z}, it follows that~$\tcX_\lm$ is smooth.
Moreover, $\dim\tcX_\lm = 5$, and hence 
$\dim \cX_\lm \le 5$.
On the other hand, since~$\cX_\lm$ is the zero locus of a section of a rank 7 vector bundle on a smooth 12-dimensional variety (see the proof of Lemma~\ref{lemma:lambda-mu-shift}),
its dimension is greater or equal than 5. Combining these two observations, we conclude that the dimension is 5 and $\cX_\lm$ is Cohen--Macaulay.

Now let us show that $\pi$ is also a blowup.
By definition 
\begin{equation}\label{eq:tcx-fiber-product}
\tcX_\lm = \LFl_\mu(3,4;W) \times_{\LGr_\mu(3,W)} \cX_\lm,
\end{equation}
and the map $\LFl_\mu(3,4;W) \to \LGr_\mu(3,W)$ is the blowup with center in $w_0(\LGr_\mu(2,\BW))$ by Lemma~\ref{lemma:sfl}, so it is enough to show that $\cX_\lm$ intersects transversally with $w_0(\LGr_\mu(2,\BW))$.
Since $\cX_\lm$ is the zero locus of the section $\lambda$ of $\cU_3^\perp(1)$ on $\LGr_\mu(3,W)$, the intersection 
\begin{equation*}
F := \cX_\lm \cap w_0(\LGr_\mu(2,\BW)) 
\end{equation*}
is the zero locus of the induced section of the vector bundle $w_0^*(\cU_3^\perp(1))$ on $\LGr_\mu(2,\BW)$. 
By~\eqref{eq:u3-bu2} we have an isomorphism $w_0^*(W/\cU_3) \cong \BW/\bcU_2$, and hence
\begin{equation*}
w_0^*\cU_3^\perp \cong \bcU_2^\perp.
\end{equation*}
Therefore $F$ is the zero locus of the induced by~$\lambda$ section of $\bcU_2^\perp(1)$ on $\LGr_\mu(2,\BW)$.
It is easy to see that this section is given by the 3-form $\lambda \conv w_0 = \blam$. 
By Lemma~\ref{lemma:gr26-3form-2form} we deduce that $F \cong \Fl(1,2;A_1) \cong \Fl(1,2;A_2)$.
In particular, $\dim(F) = 3$ and so
\begin{equation*}
\codim_{\LGr_\mu(2,\BW)}(F) = 4 = \codim_{\LGr_\mu(3,W)}(\cX_\lm),
\end{equation*}
hence the intersection is transversal.
It follows that the fiber product~\eqref{eq:tcx-fiber-product} is the blowup of $\cX_\lm$ with center in $F$.

Finally, by base change and the proof of Lemma~\ref{lemma:sfl} it follows that 
$F$ in $\cX_\lm$ is a locally complete intersection. 
Thus, smoothness of $F$ implies smoothness of $\cX_\lm$ along $F$.
On the other hand, away of $F$ the morphism $\pi$ is an isomorphism, hence $\cX_\lm$ is smooth everywhere.
%
%
%
%
\end{proof}

Theorem~\ref{theorem:main} is a combination of Proposition~\ref{proposition:blowups} and Proposition~\ref{proposition:s-z}.

\subsection{More details}\label{subsection:details}

We finish this section with some details of the geometry of the diagram~\eqref{diagram:fivefolds}.
First, note that there are three $\P^1$-bundles in the picture: $p:E \to F$, $\sigma:Z \to S$ and $\bp:\BE \to Z$.
Each of them is a projectivization of a rank 2 vector bundle
\begin{equation*}
E \cong \P_F(\cV_F),
\qquad 
Z \cong \P_S(\cV_S),
\qquad 
\BE \cong \P_Z(\cV_Z)
\end{equation*}
respectively. These vector bundles can be described explicitly.

\begin{lemma}\label{lemma:vector-bundles}
The bundles $\cV_F$ and $\cV_S$ can be represented as cohomology bundles of the monads
\begin{equation}\label{eq:cvf}
0 \to \cO(-h_1) \oplus \cO(-h_2) \xrightarrow{\ \ \ } (A_1 \oplus A_2) \otimes \cO \xrightarrow{\ \mu\ } \cO(h_1) \oplus \cO(h_2) \to 0
\end{equation} 
on $F$, and
\begin{equation}\label{eq:cvs-monad}
0 \to \cO(-\bh_1) \oplus \cO(-\bh_2) \xrightarrow{\ \ \ } (\bw2A_1 \oplus \bw2A_2) \otimes \cO \xrightarrow{\ \mu^2\ } \cO(\bh_1) \oplus \cO(\bh_2) \to 0
\end{equation} 
on $S$ respectively, and the bundle $\cV_Z$ fits into exact sequences
\begin{equation}\label{eq:cvz}
0 \to \cV_Z \xrightarrow{\ \ \ \ } \bw3\bcU_3 \oplus \bw2\bcU_3 \xrightarrow{\ \hlam\vert_Z\ } \cO(-\bh_1) \oplus \cO(-\bh_2) \to 0
\end{equation}
and
\begin{equation}\label{eq:exact-seq-cvz}
0 \to \cO(-2\sv_Z - \bh_1 - \bh_2) \to \cV_Z \to \cO(-\bh) \to \cO_{C_2}(-2) \oplus \cO_{C_1}(-2) \to 0.
\end{equation} 
\end{lemma}
\begin{proof}
By the proof of Lemma~\ref{lemma:sfl} the bundle $\cV_F$ is the kernel of the map $\mu:W/\cU_3 \to \cU_3^\vee$ on~$F$.
Moreover, by Lemma~\ref{lemma:gr26-3form} we have
\begin{equation*}
\begin{aligned}
\cU_3^\vee &\cong \cO \oplus \bcU_2^\vee &&\cong \cO \oplus \cO(h_1) \oplus \cO(h_2),\\
W/\cU_3 &\cong \BW/\bcU_2 &&\cong ((A_1 \oplus A_2) \otimes \cO) / (\cO(-h_1) \oplus \cO(-h_2)),
\end{aligned}
\end{equation*}
and the map $\mu$ factors through $\cO(h_1) \oplus \cO(h_1) \subset \cU_3^\vee$.
Altogether, this identifies $\cV_F$ with the cohomology bundle of the monad~\eqref{eq:cvf}.

Furthermore, monadic representation~\eqref{eq:cvs-monad} is just a reformulation of~\eqref{eq:cvs}.

By Lemma~\ref{lemma:blowup} and the proof of Proposition~\ref{proposition:blowups} the bundle $\cV_Z$ is the kernel of the morphism $\hlam:\bw2\cU_4 \to \cU_4^\perp$ on $Z$. 
Furthermore, as it was explained in section~\ref{subsection:hlam-z}, the map $\hlam$ factors through a surjection~\eqref{eq:hlam-z-rewritten}. This proves~\eqref{eq:cvz}.
Finally, consider the composition
$\cV_Z \hookrightarrow \bw3\bcU_3 \oplus \bw2\bcU_3 \twoheadrightarrow \bw3\bcU_3 \cong \cO(-\bh)$ 
By~\eqref{eq:cvz} its kernel and cokernel are isomorphic to the kernel and the cokernel of the map $\blam: \bw2\bcU_3 \to \cO(-\bh_1) \oplus \cO(-\bh_2)$.
So, applying~\eqref{eq:exact-seq-blam} we deduce~\eqref{eq:exact-seq-cvz}.
\end{proof}

Recall the notation for various divisorial classes introduced earlier:
\begin{itemize}
\item 
$h_1$ and $h_2$ are the hyperplane classes on $\P(A_1)$ and $\P(A_2)$;
\item 
$\bh_1$ and $\bh_2$ are the hyperplane classes on $\P(\bw2A_1)$ and $\P(\bw2A_2)$;
\item 
$\sv_E \in \Pic(E)$ is the hyperplane class on $E = \P_F(\cV_F)$;
\item 
$\sv_Z \in \Pic(Z)$ is the hyperplane class on $Z = \P_S(\cV_S)$;
\item 
$\sv_\BE \in \Pic(\BE)$ is the hyperplane class on $\BE = \P_Z(\cV_Z)$;
\item 
$h$ is the hyperplane class on $\Gr(3,W)$ and its restriction to $\cX_\lm$;
\item  
$\bh$ is the hyperplane class on $\LGr_\mu(2,\BW)_\blam$;
\item 
$e$ and $\be$ are the classes of the exceptional divisors $E$ and $\BE$ in $\tcX_\lm$.
\end{itemize}
As usually, when we have a natural map between two varieties, we suppress the pullback notation for the pullbacks of divisorial classes.


\begin{lemma}\label{lemma:picard}
We have the following relations in the Picard groups. First, we have
\begin{equation}\label{eq:pic-relations}
\begin{cases}
\bh = h - e\\
\be = h - 2e
\end{cases}
\qquad\qquad
\begin{cases}
h = 2\bh -\be\\
e = \bh - \be
\end{cases}
\qquad\qquad\text{in $\Pic(\tcX_\lm)$.}
\end{equation}
Furthermore, we have
\begin{align*}
h       &= h_1 + h_2 & \text{in $\Pic(F)$,} &&
\sv_E   &= -e & \text{in $\Pic(E)$,} \\
\bh     &= \sv_Z + \bh_1 + \bh_2 & \text{in $\Pic(Z)$,} &&
\sv_\BE &= h & \text{in $\Pic(\BE)$.}
\end{align*}
%
%
%
%
%
\end{lemma}
\begin{proof}
To prove~\eqref{eq:pic-relations} it is enough to express the classes of the exceptional divisors $e$ and $\be$ in terms of $h$ and $\bh$ using the blowup Lemma~\ref{lemma:blowup}.

To express $e$ we recall that by the argument of Lemma~\ref{lemma:sfl} the blowup $\pi$ is realized as a subvariety in $\P_{\cX_\lm}(W/\cU_3)$ corresponding to the morphism $\mu:W/\cU_3 \to \cU_3^\vee$.
Since $\det(W/\cU_3) \cong \det(\cU_3^\vee)$, it follows that $-e$ is the relative hyperplane class for this projectivization.
On the other hand, the relative hyperplane class corresponds to the line bundle $(\cU_4/\cU_3)^\vee$, that by~\eqref{eq:u4-u3} is isomorphic to $\cO(\bh - h)$.
Thus $e = h - \bh$.

Similarly, the blowup $\bpi$ is realized as a subvariety in $\P_{\LGr_\mu(3,\BW)_\blam}(\bw3\cU_4)$ corresponding to the morphism $\hlam:\bw3\cU_4 \to \cU_4^\perp$.
Since $\det(\bw3\cU_4) \cong \cO(-3\bh)$ and $\det(\cU_4^\perp) \cong \cO(-\bh)$, it follows that $2\bh - \be$ is the relative hyperplane class for this projectivization.
On the other hand, the relative hyperplane class corresponds to the line bundle $(\bw3\cU_3)^\vee \cong \cO(h)$, hence $\be = 2\bh - h$.
Now all the relations in~\eqref{eq:pic-relations} easily follow.

The relation for $h\vert_F$  follows from~\eqref{eq:u3-bu2} and~\eqref{eq:gr26-3form-bu2} and the relation for $\bh\vert_Z$ from~\eqref{eq:bcu3-1}.
Finally, since the bundle $\cV_F$ is a subbundle in $W/\cU_3$, the relative hyperplane class of $E = \P_F(\cV_F)$ equals the restriction of the relative hyperplane class of $\P_{\cX_\lm}(W/\cU_3)$, that was just shown to be equal to $-e$.
This implies $\sv_E = -e$.
Analogously, $\cV_Z$ is a subbundle in $\bw3\cU_4$, hence the relative hyperplane class of $\BE = \P_Z(\cV_Z)$ equals the restriction of the relative hyperplane class of $\P_{\LGr_\mu(3,\BW)_\blam}(\bw3\cU_4)$, that was shown to be equal to $h$. 
This implies $\sv_E = h$.
\end{proof}

Note that, the above relations show that the fibers of $Z$ over $S$ are lines on $\LGr_\mu(3,\BW)_\blam$, so $Z$ is a scroll.

\begin{lemma}\label{lemma:intersection}
We have the following relations in the Chow rings:
\begin{align*}
\sv_E^2 &= -h_1h_2 && \text{in $\CH^2(E)$},\\
\sv_Z^2 &= -\bh_1\bh_2 && \text{in $\CH^2(Z)$},\\
\sv_\BE^2 &= (3\sv_Z + 2\bh_1 + 2\bh_2)\sv_\BE - 4\sv_Z(\bh_1 + \bh_2) && \text{in $\CH^2(\BE)$}.
\end{align*}
\end{lemma}
\begin{proof}
These are just Grothendieck relations. 
So, we only have to compute the Chern classes of the bundles $\cV_F$, $\cV_S$, and $\cV_Z$.
For this we use~\eqref{eq:cvf}, \eqref{eq:cvs-monad}, and~\eqref{eq:cvz} respectively.
\end{proof}

%



\section{Applications}\label{section:apps}

\subsection{Geometry of K\"uchle fourfolds}\label{subsection:4folds}

In this section we assume that the corresponding K\"uchle fivefold $\cX_\lm$ satisfies the generality assumptions of section~\ref{subsection:genericity-assumptions}.
Recall that a K\"uchle fourfold is a hyperplane section of a K\"uchle fivefold.

Let $H_\nu \subset \P(\bw3W)$ be a hyperplane corresponding to a 3-form $\nu$ (defined up to a 3-form of type $(\lambda \conv w) + (\mu \wedge f)$, where $w \in W$ and $f \in W^\vee$).
We denote the corresponding K\"uchle fourfold by
\begin{equation*}
\Xf_\lmn := \cX_\lm \cap H_\nu \subset \cX_\lm \subset \Gr(3,W).
\end{equation*}

\begin{theorem}
Assume the pair $(\lm)$ satisfies the generality assumptions~\eqref{assumption:lambda-mu-general}, \eqref{assumption:pencil-regular}, \eqref{assumption:m-nonzero}, and~\eqref{assumption:lgr-hyperplane} and $\nu$ is a general $3$-form on $W$.
Then there is a diagram
\begin{equation}\label{diagram:4folds}
\vcenter{\xymatrix{
& 
D \ar[dl]_p \ar[r]^i & 
\tX_\lmn \ar[dl]_\pi \ar[dr]^\bpi &
\BD \ar[l]_\bi \ar[dr]^\bp 
\\
\Sigma \ar[r] & 
\Xf_\lmn &&
\bX_\lmn &
\Gamma \ar[l]
}}
\end{equation}
where
\begin{itemize}
\item 
$\Sigma = F \cap H_\nu$ is a sextic del Pezzo surface;
\item 
$\bX_\lmn \subset \LGr_\mu(3,\BW)_\blam$ is a quadratic section containing the scroll $Z$ and with singularities along a curve $\Gamma$;
\item 
$\Gamma \subset Z \subset \bX_\lmn$ is a curve of genus $37$, a section of the map $\sigma:Z \to S$ over a smooth curve in the linear system $-4K_S$ on the sextic del Pezzo surface $S$;
\item 
the map $\pi$ is the blowup with center in $\Sigma$;
\item 
the map $\bpi$ is the blowup of the Weil divisor $Z$ on $\bX_\lmn$;
\item 
$D \cong \P_\Sigma(\cV_F\vert_\Sigma)$ is the exceptional divisor of $\pi$;
\item 
$\BD \cong \P_\Gamma(\cV_Z\vert_\Gamma)$ is the exceptional locus of $\bpi$;
\item 
the maps $i$ and $\bi$ are the embeddings of $D$ and $\BD$.
\end{itemize}
\end{theorem}
%
%
\begin{proof}
We define $\tX_\lmn = \tcX_\lm \times_{\cX_\lm} \Xf_\lmn$ to be the full preimage of $\Xf_\lmn$ under the blowup $\pi:\tcX_\lm \to \cX_\lm$.
Recall that the center of this blowup is the flag variety $F \subset \P(A_1) \times \P(A_2)$ given by the pairing between $A_1$ and $A_2$ induced by the 2-form $\mu$.
Clearly, the hyperplane section $F \cap H_\nu$ is given by one more $(1,1)$-divisor in $\P(A_1) \times \P(A_2)$, corresponding to the pairing given by the 2-form $\nu \conv w_0$ on $\BW$.
So, assuming that the pencil of pairings between $A_1$ and $A_2$ given by the pencil of 2-forms $\langle \mu, \nu \conv w_0 \rangle$ is regular (this is the first genericity assumption),
we conclude that $\Sigma := F \cap H_\nu$ is a smooth del Pezzo surface of degree 6.
Moreover, in this case the intersection $F \cap H_\nu$ is transversal, hence the map $\tX_\lmn \to \Xf_\lmn$ is the blowup with center in $\Sigma$.
We denote the exceptional divisor of this blowup by $D$.
Clearly, we have
\begin{equation*}
D = \Sigma \times_F E \cong \P_\Sigma(\cV_F\vert_\Sigma).
\end{equation*}
Abusing notation we denote by $\pi$, $p$, and $i$ the maps $\tX_\lmn \to \Xf_\lmn$, $D \to \Sigma$ and $D \to \tX_\lmn$ induced by the same named maps in~\eqref{diagram:fivefolds}.

Now recall that by Lemma~\ref{lemma:picard} in the Picard group of $\tcX_\lm$ we have a relation $h = 2\bh-\be$. 
It means that the pullback to $\tcX_\lm$ of the hyperplane $H_\nu$ corresponds to a quadratic section $\bX_\lmn$ of $\LGr_\mu(3,\BW)_\blam$ containing $Z$.
Again by Lemma~\ref{lemma:picard} the restriction of $h$ to the divisor $\BE = \P_Z(\cV_Z)$ corresponds to a section of the vector bundle $\cV_Z^\vee$. 
Since $c_2(\cV_Z^\vee) = 4\sv_Z(\bh_1 + \bh_2)$ by Lemma~\ref{lemma:intersection}, it follows that for general $\nu$ this section vanishes on a curve $\Gamma \subset Z$ with the class $4\sv_Z(\bh_1 + \bh_2)$ in $\CH^2(Z)$.
The projection of the curve to $S$ is a general curve $\bGamma$ in the linear system $4(\bh_1 + \bh_2) = -4K_S$, hence of genus~37, and $\Gamma$ gives a section of $\sigma$ over $\bGamma$.
\end{proof}

The description of the K\"uchle fourfold $\Xf_\lmn$ provided by this Theorem is not as useful as the description of the fivefold in Theorem~\ref{theorem:main}, because of the singularities. 
On the other hand, it looks as if the contraction $\bpi$ is a flopping contraction, so it is interesting to consider the flop $(\tX_\lmn)^+$ of $\tX_\lmn$ and continue the two-rays game for it.

\subsection{The Chow motive of K\"uchle fivefolds}\label{subsection:motive-5folds}

For a smooth projective variety $X$ we denote by $\Mot(X)$ its Chow motive, and by $\Lef$ the Lefschetz motive.
We say that a motive is {\sf of Lefschetz type} if it is a direct sum of powers of the Lefschetz motive.
Note that we consider motives with integer coefficients.

To compute the motive of $\cX_\lm$ we use the birational transformation of Theorem~\ref{theorem:main}
and the following formula for the motive of the hyperplane section of the Lagrangian Grassmannian:
\begin{equation}\label{eq:motive-lgrh}
\Mot(\LGr_\mu(3,\BW)_\blam) = \one \oplus \Lef \oplus \Lef^2 \oplus \Lef^3 \oplus \Lef^4 \oplus \Lef^5.
\end{equation}
We prove this formula in section~\ref{section:hyperplane-sgr} by a geometric argument (see Corollary~\ref{corollary:motive-lgr-blam}), 
and now we use it to prove the following result.


\begin{theorem}\label{theorem:motive}
The Chow motive of $\cX_\lm$ is of Lefschetz type:
\begin{equation*}
\Mot(\cX_\lm) = \one \oplus \Lef \oplus (\Lef^2)^{\oplus 4} \oplus (\Lef^3)^{\oplus 4} \oplus \Lef^4 \oplus \Lef^5.
\end{equation*}
\end{theorem}
\begin{proof}
Note that the integral motive of $S$ is of Lefschetz type, since $S$ is isomorphic to a blowup of $\P^2$ in three points.
Explicitly, we have
\begin{equation*}
\Mot(S) = \one  \oplus (\Lef)^{\oplus 4} \oplus \Lef^2.
\end{equation*}
Applying the projective bundle formula to $Z \cong \P_S(\cV_S)$, we deduce
\begin{equation*}
\Mot(Z) = \Mot(S) \oplus \Mot(S) \otimes \Lef = \one  \oplus (\Lef)^{\oplus 5}  \oplus (\Lef^2)^{\oplus 5} \oplus \Lef^3.
\end{equation*}
Applying~\eqref{eq:motive-lgrh} and the blowup formula to the morphism $\bpi:\tcX_\lm \to \LGr_\mu(3,\BW)_\blam$, we deduce
\begin{multline*}
\Mot(\tcX_\lm) = \Mot(\LGr_\mu(3,\BW)_\blam) \oplus \Mot(Z) \otimes \Lef = 
\one \oplus (\Lef)^{\oplus 2} \oplus (\Lef^2)^{\oplus 6} \oplus (\Lef^3)^{\oplus 6} \oplus (\Lef^4)^{\oplus 2} \oplus \Lef^5.
\end{multline*}
On the other hand, the blowup formula for the morphism $\pi:\tcX_\lm \to \cX_\lm$ gives
\begin{equation*}
\Mot(\tcX_\lm) = \Mot(\cX_\lm) \oplus \Mot(F) \otimes \Lef.
\end{equation*}
Since a direct summand of a motive of Lefschetz type is itself a motive of Lefschetz type, and $\Mot(F) = \one  \oplus (\Lef)^{\oplus 2}  \oplus (\Lef^2)^{\oplus 2} \oplus \Lef^3$,
we deduce the required formula for $\Mot(\cX_\lm)$.
\end{proof}

Note, that a similar result for the motive with rational coefficients follows immediately from the existence of an exceptional collection (see~\cite{marcolli2015exceptional}).
However, no analogue of this result for integral coefficients is known, so some geometric argument to prove this seems necessary

The following is an immediate consequence of Theorem~\ref{theorem:motive}.

\begin{corollary}
The Hodge diamond of $\cX_\lm$ is diagonal with 
\begin{equation*}
h^{0,0}(\cX_\lm) = h^{1,1}(\cX_\lm) = h^{4,4}(\cX_\lm) = h^{5,5}(\cX_\lm) = 1,
\qquad 
h^{2,2}(\cX_\lm) = h^{3,3}(\cX_\lm) = 4,
\end{equation*}
and the Chow groups of the K\"uchle fivefold are free abelian groups
\begin{equation*}
\CH^0(\cX_\lm) \cong \CH^1(\cX_\lm) \cong \Z,  \
\CH^2(\cX_\lm) \cong \CH^3(\cX_\lm) \cong \Z^4,\
\CH^4(\cX_\lm) \cong \CH^5(\cX_\lm) \cong \Z.
\end{equation*}
\end{corollary}

One can use the description of Theorem~\ref{theorem:main} to find explicit generators of the Chow groups.






\section{A hyperplane section of the Lagrangian Grassmannian}\label{section:hyperplane-sgr}

The goal of this section is to describe the geometry of a smooth hyperplane section of the Lagrangian Grassmannian $\LGr(3,6)$.
In particular, we show that its motive is a sum of Lefschetz motives, and describe its Hilbert scheme of lines.
See~\cite{iliev2011fano} for another approach.


As usual, we denote by $\BW$ a vector space of dimension 6 and by $\mu \in \bw2\BW^\vee$ a symplectic form.
The linear span of the Lagrangian Grassmannian $\LGr_\mu(3,\BW)$ in $\P(\bw3\BW)$ is the subspace 
\begin{equation*}
\bww3\mu\BW := \Ker(\bw3\BW \xrightarrow{\ -\conv\mu\ } \BW).
\end{equation*}
The space $\bw3\BW$ has a canonical direct sum decomposition
$\bw3\BW = \bww3\mu\BW \oplus \BW$,
where the second summand is embedded by the wedge product with the bivector $\mu^{-1} \in \bw2\BW$.
Consequently, there is a direct sum decomposition
\begin{equation*}
\bw3\BW^\vee = \bww3\mu\BW^\vee \oplus \BW^\vee.
\end{equation*}
For a 3-form $\blam \in \bw3\BW^\vee$, the hyperplane section $\LGr_\mu(3,\BW)_\blam$ depends only on the projection of $\blam$ to the summand $\bww3\mu\BW^\vee$, so, we may safely assume that 
\begin{equation}
\blam \in \bww3\mu\BW^\vee.
\end{equation} 
We will keep this assumption from now on.

Consider a 7-dimensional vector space 
\begin{equation*}
W = \k \oplus \BW, 
\end{equation*}
and pack the data of the 2-form $\mu$ and the 3-form $\blam$ on $\BW$ into a single 3-form $\xi$ on $W$ as follows.
Denote by $w_0 \in W$ the base vector of the summand $\k$ and by $w_0^\vee$ a linear function on~$W$ which is zero on $\BW$ and satisfies $w_0^\vee(w_0) = 1$, and set 
\begin{equation}\label{eq:xi-def}
\xi := \blam + w_0^\vee \wedge \mu \in \bw3W^\vee.
\end{equation}
The following observation is crucial.

\begin{lemma}\label{lemma:xi-general}
If $\mu \in \bw2\BW^\vee$ is non-degenerate and $\blam \in \bww3\mu\BW^\vee$ is a $3$-form such that the hyperplane section $\LGr_\mu(3,\BW)_\blam$ of the Lagrangian Grassmannian is smooth, 
then the $3$-form $\xi \in \bw3W^\vee$ defined by~\eqref{eq:xi-def} is general.
\end{lemma}
\begin{proof}
There is a well-known necessary and sufficient condition for a 3-form $\blam \in \bww3\mu\BW^\vee$ to give a smooth hyperplane section (see~\cite{landsberg-manivel} or~\cite{iliev-ranestad}).
It is in fact equivalent to existence of a {\em Lagrangian} (with respect to the 2-form $\mu$) direct sum decomposition
\begin{equation*}
\BW = L_1 \oplus L_2
\end{equation*}
such that $\blam = \blam_1 + \blam_2$ with $\blam_i$ being generators of the subspace $\bw3L_1^\perp \subset \bww3\mu\BW^\vee$.
Choosing the bases $\{e_1,e_2,e_3\}$ and $\{e_4,e_5,e_6\}$ of the vector spaces $L_i$ appropriately, we can thus assume that
\begin{equation*}
\blam = x_{123} + x_{456}
\qquad\text{and}\qquad 
\mu = x_{16} + x_{25} + x_{34}.
\end{equation*}
Setting $e_0 = w_0$, so that $w_0^\vee = x_0$, we then have
\begin{equation*}
\xi = x_{123} + x_{456} + x_{016} + x_{025} + x_{034},
\end{equation*}
hence $\xi$ is general by Lemma~\ref{lemma:4form-general}(v).
\end{proof}

Recall that if a 3-form $\xi \in \bw3W^\vee$ is general, then its stabilizer in $\PGL(W)$ is the simple algebraic group $\Gtwo$,
and one can define its homogeneous spaces 
\begin{equation*}
\bQ_\xi \subset \P(W)
\qquad\text{and}\qquad 
\Gr_\xi(2,W) \subset \Gr(2,W).
\end{equation*} 
The quadric $\bQ_\xi$ parameterizes all vectors $w \in W$ such that the rank of the 2-form $\xi \conv w$ is less than 6,
and $\Gr_\xi(2,W)$ parameterizes all 2-subspaces $U_2 \subset W$ that are annihilated by $\xi$.

\begin{remark}\label{remark:w0-notin-gr}
If the 3-form $\xi$ is defined by~\eqref{eq:xi-def} then $w_0 \not\in \bQ_\xi$. Indeed, this is immediate since 
\begin{equation}\label{eq:xi-w0}
\xi \conv w_0 = \mu
\end{equation}
has rank 6. 
In particular, by Remark~\ref{remark:gr27-g2} the vector $w_0$ is not contained in any $U_2 \in \Gr_\xi(2,W)$.
\end{remark}

The composition $\bcU_2 \hookrightarrow \BW \otimes \cO \xrightarrow{\ \mu\ } \BW^\vee \otimes \cO \twoheadrightarrow \bcU_2^\vee$ is zero on $\LGr_\mu(2,\BW)$,
hence the composition of the first two arrows factors through the subbundle $\bcU_2^\perp \subset \BW^\vee \otimes \cO$. 
This allows to consider $\bcU_2$ as a subbundle in $\bcU_2^\perp$, and to define the quotient bundle $\bcU_2^\perp/\bcU_2$.
Analogously, the composition $\bw2\bcU_2 \hookrightarrow \bw2\BW \otimes \cO \xrightarrow{\ \blam\ } \BW^\vee \otimes \cO \twoheadrightarrow \bcU_2^\vee$ is zero on the hyperplane section $\LGr_\mu(2,\BW)_\blam$,
hence the composition of the first two arrows again factors through the subbundle $\bcU_2^\perp \subset \BW^\vee \otimes \cO$. 
Therefore, the composition
\begin{equation*}
\cO(-1) \cong \bw2\bcU_2 \xrightarrow{\ \blam\ } \bcU_2^\perp \twoheadrightarrow \bcU_2^\perp/\bcU_2
\end{equation*}
defines a global section of the vector bundle $(\bcU_2^\perp/\bcU_2)(1)$ on $\LGr_\mu(2,\BW)$. 
We denote by
\begin{equation*}
D_{\blam,\mu} \subset \LGr_\mu(2,\BW)
\end{equation*}
its zero locus.

\begin{proposition}\label{proposition:dlm-grxi}
Let $W = \k \oplus \BW$ and let $\xi$ be the $3$-form on $W$ defined by~\eqref{eq:xi-def}.
If $\xi$ is general then
\begin{equation*}
D_{\blam,\mu} \cong \Gr_\xi(2,W)
\end{equation*}
is the $\Gtwo$-adjoint variety corresponding to the $3$-form $\xi$. In particular, $D_{\blam,\mu}$ is a smooth fivefold.
\end{proposition}
\begin{proof}
Consider the embedding $w_0:\LGr_\mu(2,\BW) \to \Gr(3,W)$ defined by taking a subspace $\BU_2 \subset \BW$ to the subspace $U_3 = \k w_0 \oplus \BU_2 \subset W$.
It follows from~\eqref{eq:xi-w0} that the image lies in the hyperplane section of $\Gr(3,W)$ given by the 3-form $\xi$
and in fact is identified with the subvariety of this hyperplane section $\Gr(3,W)_\xi$ parameterizing all 3-subspaces containing $w_0$.

Define on $\Gr(3,W)_\xi$ a morphism
\begin{equation*}
\hxi:\bw2\cU_3 \to \cU_3^\perp,
\end{equation*}
analogously to the definition of the morphism $\hlam$ in~\eqref{eq:map-hlam}. 
We restrict this morphism to $\LGr_\mu(2,\BW)$ and consider its degeneracy loci $D_k(\hxi) \subset \LGr_\mu(2,\BW)$.
We will show that $D_2(\hxi) = \emptyset$ and $D_1(\hxi) \cong \Gr_\xi(2,W)$.

Indeed, assume first that $U_3 \subset W$ is a $\xi$-isotropic subspace containing $w_0$ such that the corank of the map $\hxi$ at~$U_3$ is~1.
This means that the map $\hxi:\bw2U_3 \to U_3^\perp$ is not injective. 
Then its kernel is generated by a bivector in $U_3$, which is necessarily decomposable, and thus corresponds to a 2-subspace $U_2 \subset U_3$.
Then the condition that $\bw2U_2 \subset \bw2U_3$ is in the kernel of $\hxi$ means that $\xi$ annihilates~$U_2$.
Therefore, $U_2$ is a point of the $\Gtwo$-adjoint variety $\Gr_\xi(2,W)$.
Since $U_3$ also contains $w_0$ and $w_0 \not\in U_2$ by Remark~\ref{remark:w0-notin-gr}, it follows that
\begin{equation*}
U_3 = \k w_0 \oplus U_2.
\end{equation*}
Thus $D_1(\hxi)$ parameterizes all subspaces representable as the sum of the line generated by $w_0$ and a 2-subspace $U_2$ annihilated by $\xi$.
In particular, $D_1(\hxi)$ is equal to the the image of the regular map $\Gr_\xi(2,W) \to \Gr(2,\BW)$ induced by the linear projection $\pr:W \to \BW$ along $w_0$.

Now assume that the corank of $\hxi$ at $U_2$ is 2 or more. 
Then the subspace $U_3$ contains a pencil of 2-dimensional subspaces $U_2 \subset U_3$ annihilated by $\xi$.
Clearly, at least one of subspaces in the pencil contains $w_0$, which contradicts Remark~\ref{remark:w0-notin-gr}.
This shows that $D_2(\hxi) = \emptyset$. Moreover, this also shows that for $U_3 \in D_1(\hxi)$ the subspace $U_2 \subset U_3$ annihilated by $\xi$ is unique,
hence the map $\pr:\Gr_\xi(2,W) \to D_1(\hxi) \subset \LGr_\mu(2,\BW)$ is an isomorphism.

Now, finally, let us show that $D_1(\hxi) = D_{\blam,\mu}$. Indeed, assume $\BU_2 \subset \BW$ is a 2-subspace, such that $\hxi$ is degenerate at $U_3 = \k w_0 \oplus \BU_2$.
If $\{u',u''\}$ is a basis of $\BU_2$ then $\{w_0\wedge u', w_0\wedge u'', u' \wedge u''\}$ is a basis of $\bw2U_3$ and by~\eqref{eq:xi-def} we have
\begin{equation*}
\begin{aligned}
\xi \conv (w_0\wedge u') &= \mu \conv u',\\
\xi \conv (w_0\wedge u'') & = \mu \conv u'',\\
\xi \conv (u' \wedge u'') &= \blam \conv (u' \wedge u'') + \mu(u',u'')w_0^\vee.
\end{aligned}
\end{equation*}
Since the 2-form $\mu$ is nondegenerate, the two linear functions $\mu\conv u'$ and $\mu \conv u''$ are linearly independent.
Moreover, they vanish on $w_0$.
Hence the degeneracy condition means that $\mu(u',u'') = 0$ and the linear function $\blam \conv (u' \wedge u'')$ is a linear combination of $\mu \conv u'$ and $\mu \conv u''$.
The first means that $\BU_2$ is $\mu$-isotropic, and the second means that $\blam$, considered as a section of $(\bcU_2^\perp/\bcU_2)(1)$, vanishes at $\BU_2$.
Thus $D_1(\xi) \subset D_{\blam,\mu}$. 
The converse statement can be proved by the same computation.
%
\end{proof}

Now we are ready to give an alternative geometric description of $\LGr_\mu(3,\BW)_\blam$.
Consider the relative Grassmannian
\begin{equation*}
\P_{\LGr_\mu(3,\BW)_\blam}(\bw2\bcU_3) \cong \Gr_{\LGr_\mu(3,\BW)_\blam}(2,\bcU_3) \cong \LFl_\mu(2,3;\BW) \times_{\LGr_\mu(3,\BW)} \LGr_\mu(3,\BW)_\blam.
\end{equation*}
Being embedded into the symplectic flag variety, it comes with the tautological flag of vector bundles $\bcU_2 \hookrightarrow \bcU_3 \hookrightarrow \BW \otimes \cO$.
We denote by
\begin{equation}\label{diagram:gr-lgr}
\vcenter{\xymatrix{
 & \Gr_{\LGr_\mu(3,\BW)_\blam}(2,\bcU_3) \ar[dl]_{\rho_2} \ar[dr]^{\rho_3} \\
\LGr_\mu(2,\BW) && \LGr_\mu(3,\BW)_\blam 
}}
\end{equation}
the natural projections, induced by the projections of the symplectic flag variety. 

\begin{theorem}\label{theorem:lgrh}
The map $\rho_3$ is a $\P^2$-bundle. 
If the hyperplane section $\LGr_\mu(3,\BW)_\blam$ of $\LGr_\mu(3,\BW)$ is smooth, then
the map $\rho_2:\Gr_{\LGr_\mu(3,\BW)_\blam}(2,\bcU_3) \to \LGr_\mu(2,\BW)$ is the blowup with center in the subvariety $D_{\blam.\mu} \cong \Gr_\xi(2,W)$.
%
%
\end{theorem}
\begin{proof}
The first part is evident. For the second part, 
we use again the blowup Lemma~\ref{lemma:blowup}. 
We note that the symplectic flag variety $\LFl_\mu(2,3;\BW)$ can be represented as the projectivization of the (self-dual) vector bundle $\bcU_2^\perp/\bcU_2$ on $\LGr_\mu(2,\BW)$,
and that the subvariety $\Gr_{\LGr_\mu(3,\BW)_\blam}(2,\bcU_3) \subset \LFl_\mu(2,3;\BW)$ is the zero locus of the global section of the line bundle $\cO(\bh) \cong \rho_3^*\cO(1)$,
which is a hyperplane class for the projection $\LFl_\mu(2,3;\BW) \to \LGr_\mu(2,\BW)$, corresponding to the twisted bundle $(\bcU_2^\perp/\bcU_2)(1)$. 
Therefore, the 3-form $\blam$ gives a global section of this bundle, and it is easy to see that this section is the same as the section, defining the subvariety $D_{\blam,\mu} \subset \LGr_\mu(2,\BW)$.

If the hyperplane section $\LGr_\mu(3,\BW)_\blam$ is smooth, then by Lemma~\ref{lemma:xi-general} the 3-form $\xi$ defined by~\eqref{eq:xi-def} is general,
and hence by Proposition~\ref{proposition:dlm-grxi} the zero locus $D_{\blam,\mu}$ of $\blam$ on $\LGr_\mu(2,\BW)$ is isomorphic to the $\Gtwo$-adjoint variety $\Gr_\xi(2,W)$.
Its codimension in $\LGr_\mu(2,\BW)$ is $7-5 = 2$, hence Lemma~\ref{lemma:blowup} applies and shows that the map $\rho_2$ is the blowup with center in $\Gr_\xi(2,W)$.
%
\end{proof}

This Theorem has the following nice consequences.

\begin{corollary}\label{corollary:motive-lgr-blam}
The Chow motive of $\LGr_\mu(3,\BW)_\blam$ is of Lefschetz type:
\begin{equation*}
\Mot(\LGr_\mu(3,\BW)_\blam) = \one \oplus \Lef \oplus \Lef^2 \oplus \Lef^3 \oplus \Lef^4 \oplus \Lef^5.
\end{equation*}
\end{corollary}
\begin{proof}
Note that both $\LGr_\mu(2,\BW)$ and $\Gr_\xi(2,W)$ are homogeneous spaces (for the groups $\Sp(W)$ and $\Gtwo$ respectively), 
hence their Chow motives are of Lefschetz type:
\begin{align*}
\Mot(\LGr_\mu(2,\BW)) &= \one \oplus \Lef \oplus (\Lef^2)^{\oplus 2} \oplus (\Lef^3)^{\oplus 2} \oplus (\Lef^4)^{\oplus 2} \oplus (\Lef^5)^{\oplus 2} \oplus \Lef^6 \oplus \Lef^7,\\
\Mot(\Gr_\xi(2,W)) &= \one \oplus \Lef \oplus \Lef^2 \oplus \Lef^3 \oplus \Lef^4 \oplus \Lef^5.
\end{align*}
By the blowup formula, we have
\begin{equation*}
\Mot(\Gr_{\LGr_\mu(3,\BW)_\blam}(2,\bcU_3)) = \one \oplus (\Lef)^{\oplus 2} \oplus (\Lef^2)^{\oplus 3} \oplus (\Lef^3)^{\oplus 3} \oplus (\Lef^4)^{\oplus 3} \oplus (\Lef^5)^{\oplus 3} \oplus (\Lef^6)^{\oplus 2} \oplus \Lef^7.
\end{equation*}
On the other hand, by the projective bundle formula we have
\begin{equation*}
\Mot(\Gr_{\LGr_\mu(3,\BW)_\blam}(2,\bcU_3)) = \Mot(\LGr_\mu(3,\BW)_\blam) \otimes (\one \oplus \Lef \oplus \Lef^2).
\end{equation*}
In particular, $\Mot(\Gr_{\LGr_\mu(3,\BW)_\blam}(2,\bcU_3))$ is a direct summand of a motive of Lefschetz type, hence is itself a motive of Lefschetz type.
The explicit form of its decomposition easily follows.
\end{proof}

\begin{remark}
Of course, it follows from Corollary~\ref{corollary:motive-lgr-blam} that the Chow groups of $\LGr_\mu(3,\BW)_\blam$ are free abelian of rank 1:
\begin{equation*}
\CH^p(\LGr_\mu(3,\BW)_\blam) \cong \Z,
\qquad
0 \le p \le 5.
\end{equation*}
Moreover, one can find explicit generators of those groups --- these are the fundamental class, the class of a hyperplane section, $c_2(\bcU_3)$, $c_3(\bcU_3)$, the class of a line, and the class of a point.
An interesting feature of this example is that Chern classes of the full exceptional collection do not generate the Chow ring (the class of a line is not generated).
\end{remark}

Another immediate application of the geometric construction of Theorem~\ref{theorem:lgrh} is the following description of the Hilbert scheme of lines.

\begin{corollary}
The Hilbert scheme of lines on $\LGr_\mu(3,\BW)_\blam$ is isomorphic to $\Gr_\xi(2,W)$.
Moreover, the exceptional divisor of the blowup $\rho_2$ in~\eqref{diagram:gr-lgr} is the universal family of lines.
\end{corollary}
\begin{proof}
The Hilbert scheme of lines on the Lagrangian Grassmannian $\LGr_\mu(3,\BW)$ is well-known to be isomorphic to the isotropic Grassmannian $\LGr_\mu(2,\BW)$
with the universal family of lines provided by the isotropic flag variety
\begin{equation*}
\LFl_\mu(2,3;\BW) \cong \P_{\LGr_\mu(2,\BW)}(\bcU_2^\perp/\bcU_2).
\end{equation*}
It follows that the Hilbert scheme of lines on the hyperplane section $\LGr_\mu(3,\BW)_\blam$ given by a 3-form $\blam$ is the zero locus of the section of $(\bcU_2^\perp/\bcU_2)(1)$ given by $\blam$,
i.e.\ coincides with the subscheme $D_{\blam,\mu} \subset \LGr_\mu(2,\BW)$. 
So, Proposition~\ref{proposition:dlm-grxi} applies.
\end{proof}

\begin{remark}
One can modify the construction of this section to get a birational description of a hyperplane section $\LGr_\mu(3,\BW)_\blam$ of the Lagrangian Grassmannian.
Consider a linear function $f \in \BW^\vee$ and the zero locus $M$ of $f$, considered
as a global section of $\bcU_2^\vee$ on $\Gr_{\LGr_\mu(3,\BW)_\blam}(2,\bcU_3)$. 
Then it is easy to see that the projection $\rho_3$ maps $M$ to $\LGr_\mu(3,\BW)_\blam$ birationally (and in fact identifies $M$ with the blowup of $\LGr_\mu(3,\BW)_\blam$ with center in a quadric surface),
and the projection $\rho_2$ maps $M$ birationally onto a hyperplane section $\Gr_\mu(2,5)$ of $\Gr(2,5)$ (and in fact identifies $M$ with the blowup of $\Gr_\mu(2,5)$ with center in the flag variety $\Fl(1,2;3)$).
This construction gives an alternative way to describe the motive of $\LGr_\mu(3,\BW)_\blam$ from the relation
$\Mot(\LGr_\mu(3,\BW)_\blam) \oplus \Mot(Q^2) \otimes (\Lef \oplus \Lef^2) = \Mot(\Gr_\mu(2,5)) \oplus \Mot(\Fl(1,2;3)) \otimes \Lef$.
\end{remark}


\providecommand{\bysame}{\leavevmode\hbox to3em{\hrulefill}\thinspace}
\providecommand{\MR}{\relax\ifhmode\unskip\space\fi MR }
\providecommand{\MRhref}[2]{%
  \href{http://www.ams.org/mathscinet-getitem?mr=#1}{#2}
}
\providecommand{\href}[2]{#2}

\end{document}